\documentclass[reqno]{amsart}
\usepackage{amssymb, latexsym, amsmath, amsfonts, amscd}
\usepackage{lscape,color}
\usepackage{mathrsfs, hyperref}

\newtheorem{lemma}{Lemma}
\newtheorem{theorem}{Theorem}
\newtheorem{proposition}{Proposition}
\newtheorem{corollary}{Corollary}

{\theoremstyle{definition}}
{\theoremstyle{definition}}
{\theoremstyle{definition}\newtheorem*{remark}{Remark}}
{\theoremstyle{definition}}

\numberwithin{lemma}{section}

\newcommand{\R}{\mathbf{R}}

\newcommand{\N}{\mathbf{N}}

\newcommand{\sech}{\textnormal{sech}}
\newcommand{\p}{\partial}

\newcommand{\supp}{\textnormal{supp }}
\newcommand{\tw}{\tilde{w}}
\newcommand{\tv}{\tilde{v}}
\newcommand{\la}{\langle}

\title{Asymptotic Stability for KdV Solitons in Weighted $H^s$ Spaces}
\author[B. Pigott]{Brian Pigott}
\address{Wofford College\\pigottbj@wofford.edu}

\author[S. Raynor]{Sarah Raynor}
\address{Wake Forest University\\raynorsg@wfu.edu}
\date{\today}
\thanks{B. Pigott completed this work while a Teacher-Scholar Postdoctoral Fellow at Wake Forest University would like to thank the Department of Mathematics at Wake Forest University for its support.}
\thanks{S. Raynor would like to thank the Simons Foundation for their support during the creation of this work.}

\begin{document}

\begin{abstract}
In this work, we consider the stability of solitons for the KdV equation below the energy space, using spatially-exponentially-weighted norms.  Using a combination of the $I$-method and spectral analysis following Pego and Weinstein, we are able to show that, in the exponentially weighted space, the perturbation of a soliton decays exponentially for arbitrarily long times.  The finite time restriction is due to a lack of global control of the unweighted perturbation.  
\end{abstract}

\maketitle

\section{Introduction}

Consider the initial value problem for the Korteweg-de Vries equation (KdV)
\begin{equation}
\label{kdv}
\left \{ \begin{array}{l l}
u_t+u_{xxx}+\p_x(u^2)=0, & \\
u(0,x) = u_{0}(x).\\
\end{array}
\right .
\end{equation}
This is a well-known nonlinear dispersive partial differential equation modelling the behavior of water waves in a long, narrow, shallow canal.  Of particular interest are soliton solutions to this equation, which are special travelling wave solutions of the form
\begin{equation}\label{soliton}
Q_{c,x_0}(x,t)=\psi_c(x-ct-x_0)=\frac{3c}{2}\sech^2(\frac{\sqrt{c}}{2}(x-ct-x_0)).
\end{equation}
The stability of these solitons has been an area of intense study for many years.  One might first be interested in the orbital stability of the soliton.  That is, if $u_{0}(x)-\psi_c(x)$ is small in an appropriate norm, then, for all time there is some $x_0(t)$ so that $u(x,t)-\psi_c(x-x_0(t))$ remains small.  The study of orbital stability in the energy space $H^1$ began with with Benjamin \cite{MR0338584} and Bona \cite{MR0386438}; see also \cite{MR897729}. This work was made systematic by Weinstein \cite{MR820338}, who established the orbital stability of solitons for nonlinear Schr\"{o}dinger equations and for generalized KdV equations.  One can also study the possibility of orbital stability of solitons in $H^s$ for $s$ not an integer, and in \cite{MR1995945} and \cite{MR3082567} it was shown that, for $0<s<1$, the possible orbital instability of the solitons is at most polynomial in time.

Also of interest is the concept of asymptotic stability, meaning that there exist $c_+$ and $x_+$ so that, in some appropriate sense, $u(x,t)-\psi_{c_+}(x-c_+t-x_+)$ goes to zero as time goes to positive infinity.  Asymptotic stability for the Korteweg-deVries equation was first studied by Pego and Weinstein in \cite{MR1289328}.  In that paper, the authors considered the behavior of solutions to KdV in the weighted space $H^1_a=\{f|\|e^{ax}f(x)\|_{H^1} < + \infty\}$, for appropriate choice of $a$.  In that setting, they were able to conclude that solitons are asymptotically stable and, in fact, converge exponentially to the limiting soliton.   Asymptotic stability in the space $H^1$ was established by Martel and Merle in \cite{MR1826966,MR2109467,MR2385662}, and in $L^2$ by Merle and Vega \cite{MR1949297} via the Miura transform. More recently, Mizumachi and Tzvetkov \cite{mizumachi-tzvetkov} have adapted arguments from \cite{MR1289328} to establish asymptotic stability for KdV solitons in $L^{2}$, with exponential rate of approach in an exponentially weighted space.

In this paper, we consider the case of asymptotic stability in $H^s$, $0 < s < 1$.  It may seem clear that asymptotic stability in $L^2$ and $H^1$ should imply the same in the spaces $H^s$, $0<s<1$, but this is not the case.  The natural interpolation does not work because $H^s$ functions are not in $H^1$.  Another natural technique to consider is the well-known $I$-method of Colliander, Keel, Staffilani, Takaoka, and Tao.  This has been done, for KdV in \cite{MR1995945} and \cite{MR3082567} and for the nonlinear Schr\"{o}dinger equation in \cite{MR1961447,MR1951312}.  However, the $I$-method naturally loses an error term which amounts to polynomial growth in time of the computed perturbation.  We note that this is an artifact of the technique, and is not believed to be a real property of solutions to KdV.  

Our goal here is to remove the polynomial loss of control of the perturbation.  To do so, we reconsider the exponentially weighted spaces of Pego and Weinstein.  We establish local well-posedness for the exponentially weighted soliton perturbation in a space $X^{s,1/2,1}$ which embeds into the Bourgain space $X^{s,1/2+}$, partially following the local well-posedness work of Molinet and Ribaud \cite{MR1889080,MR1918236}, and Guo and Wang \cite{MR2514729} on dispersive-dissipative equations. In so doing we establish multilinear estimates that accommodate the presence of the exponential weight. For technical reasons, this requires that $s > 7/8$.  We then use the $I$-method to map our solutions into an exponentially-weighted version of $H^1$.  Finally, we run an iteration scheme to establish global control of the perturbation in $H^s$ and the exponentially weighted space $H^s_a$, concluding that the soliton is exponentially asymptotically stable in $H^s_a$ for $s > 7/8$. Specifically, we show the following:

\begin{theorem}\label{main}
There exist $\epsilon_1 > 0$ and $0 <r < 1$ and for every $T>0$ there exists $\epsilon_2 >0$ so that if $\|e^{ay}I_1v(0)\|_{H^1} < \epsilon_1$, $|c(0)-c_0| < \epsilon_1$  and $\|I_1v(0)\|_{H^1}< \epsilon_2$, then there exist piecewise differentiable functions $c(t)$, $\gamma(t)$ and a constant $C> 0$ so that for all $ t \in [0,T]$:
\begin{enumerate}
\item $\|e^{ay}I_1v(t)\| < C \epsilon_1r^t$,
\item $|\dot{c}| + |\dot{\gamma}| < C\epsilon_1r^t$, and 
\item $|c(t)-c_0| < 2C\epsilon_1$.
\end{enumerate}
\end{theorem}

The paper is organized as follows: In section 2, we will set up our notation and establish basic results.  In section 3, we will establish some necessary estimates to establish local well-posedness in section 4.  In section 5, we will run the iteration scheme and establish the main result of the paper.

\section{Notation and Basic Results}

We will define the Fourier multiplier operator $I_N$ by $\widehat{I_Nf}(\xi)=m_N(\xi)\hat{f}(\xi)$, with $m_N$ a smooth, even, decreasing function of $|\xi|$ which satisfies $m_N(\xi)=1$ for $|\xi|<N$ and $m_N(\xi)=\frac{|\xi|^{s-1}}{N^{s-1}}$ for $|\xi| > 10N$.  In this paper, $N$ will be a function of our time-step $n$, and, in particular 
\begin{equation*}
N(n) = \kappa^{\left (-\frac{1}{\frac{7}{4}-s}+\eta_1 \right )n}
\end{equation*}
for $\eta_1 >0 $ very small, where  $1 > \kappa > \sqrt{1-\frac{b}{2}}$, and $b$ are defined below. 

We define $\tilde{v}_n(t) = I_{N(n)}v(y,t)$ and $\tilde{w}_n(t)=e^{ay}I_{N(n)}v(y,t)$, where $y= x-\int_0^tc(s)ds-\gamma(t)$, and $c(t)$, $\gamma(t)$ are chosen so that, at each time $t$, for appropriate value of $n$, $\|\tilde{w}_n(t)\|_{L^2}$ is minimized.   In order to do so, we first need to consider the difference equations satisfied by $\tilde{v}$ and $\tilde{w}$, and consider their linearizations about the soliton.

\begin{lemma} The perturbation $\tilde{v}$ satisfies the difference equation
\begin{equation}
\label{vequation}
\begin{aligned}
(\tilde{v}_n)_t &= \p_y(-\p_y^2+c_0-2\psi_c)\tilde{v}_n + I_{N(n)}\p_y({v}^2) + \p_y(I_{N(n)}(\psi_cv)-\psi_cI_{N(n)}v) \\ \qquad & + (\dot{\gamma}\p_y + \dot{c}\p_c)I_{N(n)}\psi_c + (\dot{\gamma}+c-c_0)\p_y\tilde{v}
\end{aligned}
\end{equation}
Moreover, the perturbation $\tilde{w}_n(t)$ satisfies the difference equation
\begin{equation}\label{wequation}
\begin{aligned}
(\tilde{w}_n)_t &= e^{ay}\p_y(-\p_y^2+c_0-2\psi_c)e^{-ay}\tilde{w}_n+(c-c_0-\dot{\gamma})(\p_y-a)\tilde{w}\\   -&e^{ay}I_{N(n)}\p_y(v^2)-e^{ay} (\dot{c}\p_c + \dot{\gamma}\p_y)I_{N(n)}\psi_c-e^{ay}\p_y(I_{N(n)}(\psi_cv)-\psi_cI_{N(n)}v)
\end{aligned}
\end{equation}
\end{lemma}
\begin{proof} From \cite{MR1289328}, we have that 
\begin{equation*}
v_t=p_y(-\p_y^2+c_0-2\psi_c)v + \p_y(v^2) + (\dot{\gamma}\p_y + \dot{c}\p_c)\psi_c + (\dot{\gamma}+c-c_0)\p_yv
\end{equation*}
and
\begin{align*}
w_t&=e^{ay}\p_y(-\p_y^2+c_0-2u_c)e^{-ay}w + (c-c_0)(\p_y-a)w + [e^{ay}(\dot{c}\p_c + \dot{\gamma}\p_y]u_c \\ \qquad &+ \dot{\gamma}(\p_y-a)w + e^{ay}\p_y(c-c_0 + v^2)e^{-ay}w].
\end{align*}
The result here comes from applying $I$ to each equation.
\end{proof}

For fixed $c>0$, define the operator $A_a=e^{ay}\p_y(-\p_y^2+c-2\psi_c)e^{-ay}.$ We have the following from \cite{MR1289328},\cite{pigottraynorH1}:
\begin{proposition}\label{spectral} For $0 < a < \sqrt{\frac{c}{3}},$ the spectrum of $A_a$ in $H^1$ consists of the following:
\begin{enumerate}
\item An eigenvalue of algebraic multiplicity $2$ at $\lambda=0$.  A generator of the kernel of $A_a$ is $\zeta_1=e^{ay}\p_y\psi_c$, and the second generator of the generalized kernel of $A_a$ is $\zeta_2=e^{ay}\p_c\psi_c$.
\item A continuous spectrum $S^a$ parametrized by $\tau \to i\tau^3-3a\tau^2+(c-3a^2)i\tau-a(c-a^2)$.  For any element $\lambda$ of this continuous spectrum, the real part of $\lambda$ is at most $b:=-a(c-a^2)<0$.
\end{enumerate}
The spectrum contains no other elements.
\end{proposition}
We also need to consider the elements of the spectrum to $A_a^*$, which are \\$\eta_1=e^{-ay}[\theta_1\p_y^{-1}\p_c\psi_c + \theta_2\psi_c]$ and $\eta_2=e^{-ay}(\theta_3\psi_c)$, where $\p_y^{-1}f$ is defined to be $\int_{-\infty}^yf(t)dt$ and $\theta_1$, $\theta_2$ and $\theta_3$ are appropriate constants to obtain the biorthogonality relationship $\langle \zeta_j,\eta_k\rangle=\delta_{jk}$.  We will define the $L^2$ spectral projections $Pw=\sum_{i=1}^2\langle w, \eta_i\rangle \zeta_i$ and $Qw=w-Pw$ onto the discrete and continuous spectrums of $A_a$ respectively, with respect to the fixed initial value of $c$, $c_0$. 

Returning to the difference equation \eqref{wequation}, for each fixed $t$ we select $\dot{c}_n(t)$ and $\dot{\gamma}_n(t)$ so that $P\tilde{w}_n=0$, and $Q\tilde{w}_n=\tilde{w}_n$.
Defining ${\tilde{\mathcal{F}}}=(c-c_0-\dot{\gamma})(\p_y-a)\tilde{w}-e^{ay}I_{N(n)}\p_y(v^2)-e^{ay} (\dot{c}\p_c + \dot{\gamma}\p_y)I_{N(n)}\psi_c-e^{ay}\p_y(I_{N(n)}(\psi_cv)-\psi_cI_{N(n)}v),$ and ${\tilde{\mathcal{G}}}= (c-c_0)(\p_y-a)\tilde{w}-e^{ay}I_{N(n)}\p_y(v^2) -e^{ay}\p_y(I_{N(n)}(\psi_cv)-\psi_cI_{N(n)}v)$ we have that
$$w_t = A_aw + Q{\tilde{\mathcal{F}}},$$ and
\begin{equation}\label{modulation}
\mathcal{A}\begin{bmatrix} \dot{\gamma} \\ \dot{c} \end{bmatrix} = \begin{bmatrix} \langle {\tilde{\mathcal{G}}},\eta_1 \rangle \\ \langle{\tilde{\mathcal{G}}}, \eta_2 \rangle \end{bmatrix},
\end{equation}
where $\mathcal{A}$ is the matrix
\begin{equation*}
\mathcal{A} = \begin{bmatrix} 1 + \langle e^{ay}(\p_y\psi_c-\p_y\psi_{c_0}),\eta_1 \rangle -\langle \tilde{w}_n, \p_y\eta_1 \rangle & \langle e^{ay}(\p_c\psi_c-\p_c\psi_{c_0}),\eta_1 \rangle \\ \langle  e^{ay}(\p_y\psi_c-\p_y\psi_{c_0}),\eta_2 \rangle-\langle \tilde{w}_N, \p_y\eta_2 \rangle & 1 + \langle e^{ay}(\p_c\psi_c-\p_c\psi_{c_0}),\eta_2 \rangle \end{bmatrix}.
\end{equation*}

\section{Linear and Multilinear Estimates}

In this section we will review the construction of the space $X^{s,1/2,1}$ and mention the linear estimates which were developed in \cite{pigottraynorH1}.  At the end of this section we prove a new bilinear estimate which is then used to establish a multilinear estimate that is necessary for the proof of Theorem \ref{main}. 

First, we provide a version of the product rule that holds with the multiplier operator $I$ in place of a derivative:
\begin{lemma}\label{expIproductrule}
Suppose that $\|e^{ay}f_i\|_{L^2} < \infty$ and $\|I_N\p_yf_i\|_{L^2}< \infty$ for $i=1,2$.  Then
$$\|e^{ay}I_N\p_y(f_1f_2)\|_{L^2}\leq 2\|I_Nf_1\|_{H^1}\|e^{ay}I_N\p_yf_2\|_{L^2} + 2\|I_Nf_2\|_{H^1}\|e^{ay}I_N\p_yf_1\|_{L^2}.$$
\end{lemma}
\begin{proof} Define $\omega_R(y)=\chi_{\{y \leq R\}}e^{ay}$, and consider $\|\omega_RI_N\p_y(f_1f_2)\|_{L^2}$.  Taking the Fourier transform and using duality, we find that this equals $$\sup_{\|f\|_{L^2} = 1} \int\int\int\limits_{\Gamma_{4}} \widehat{\omega_R}(\xi_1)m(\xi_2+\xi_3)(\xi_2+\xi_3)\hat{f}_1(\xi_2)\hat{f}_2(\xi_3)f(\xi_4),$$
where $\Gamma_{4} = \{ (\xi_{1}, \xi_{2}, \xi_{3}, \xi_{4}) \in \R^{4} \ \vert \ \xi_{1} + \xi_{2} + \xi_{3} + \xi_{4} = 0 \}$. 
Now, either $\xi_2+\xi_3 \leq 2 \xi_2$ or $\xi_2 + \xi_3 \leq 2 \xi_3$.  In the first case, note that $m(\xi_2 + \xi_3) (\xi_2+\xi_3) \leq 2 m(\xi_2)\xi_2$ by the properties of $m$, so we have, with $\xi_5=\xi_2+\xi_3$ and $\xi_6=\xi_1+\xi_5$,
\begin{align*}
\|\omega_RI_N\p_y(f_1f_2)\|_{L^2} & \leq 2\sup_{\|f\|_{L^2} = 1} \int\int\int\limits_{\Gamma_{4}} \widehat{\omega_R}(\xi_1)m(\xi_2)\xi_2\hat{f}_1(\xi_2)\hat{f}_2(\xi_3)f(\xi_4)\\
& = 2\sup_{\|f\|_{L^2} = 1} \int\int\int\limits_{\Gamma_{4}} \widehat{\omega_R}(\xi_1)\widehat{(I\p_yf_1)}(\xi_2)\hat{f}_2(\xi_3)f(\xi_4)\\
& = 2\sup_{\|f\|_{L^2} = 1} \int \int\limits_{\xi_1+\xi_5+\xi_4=0}\hat{f}_2(\xi_5)(\omega_R I\p_yf_1)\widehat{\ \ }(\xi_5)f(\xi_4)\\
& = 2\sup_{\|f\|_{L^2}=1}\int_{\xi_6+\xi_4=0}(f_2\omega_RI\p_yf_1)\widehat{\ \ }(\xi_6)f(\xi_4)\\
& \leq 2 \sup_{\|f\|_{L^2} = 1} \|f_2\omega_RI\p_yf_1\|_{L^2}\|f\|_{L^2}\\
& =2\|f_2\omega_RIp_yf_1\|_{L^2}\\
& \leq 2 \|f_2\|_{L^\infty}\|\omega_RIp_yf_1\|_{L^2}\\
& \leq 2\|f_2\|_{H^s}\|e^{ay}Ip_yf_1\|_{L^2}\\
& \leq 2\|I_Nf_2\|_{H^1}\|e^{ay}Ip_yf_1\|_{L^2}.
\end{align*}
By the symmetry between the two cases, we obtain in total that
$$\|\omega_RI_N\p_y(f_1f_2)\|_{L^2}  \leq  2\|I_Nf_2\|_{H^1}\|e^{ay}Ip_yf_1\|_{L^2}+ 2\|I_Nf_1\|_{H^1}\|e^{ay}Ip_yf_2\|_{L^2}.$$
Now, letting $R \to \infty$, since $\chi_{\{y < R\}}|e^{ay}I_Np_y(f_1f_2)(y)|^2$ is a pointwise-increasing function in $R$, by the Lebesgue monotone convergence theorem we see that
\begin{align*}
\|e^{ay}I_N\p_y(f_1f_2)\|_{L^2}^2 & = \int |e^{ay}I_Np_y(f_1f_2)(y)|^2dy \\
& = \lim_{R \to \infty} \int \chi_{\{y < R\}}|e^{ay}I_Np_y(f_1f_2)(y)|^2 \\
& = \lim_{R\to\infty}\|\omega_RI_N\p_y(f_1f_2)\|_{L^2}^2\\
& \leq \lim_{R\to\infty} (2\|I_Nf_2\|_{H^1}\|e^{ay}Ip_yf_1\|_{L^2}+ 2\|I_Nf_1\|_{H^1}\|e^{ay}Ip_yf_2\|_{L^2})^2\\
& = (2\|I_Nf_2\|_{H^1}\|e^{ay}Ip_yf_1\|_{L^2}+ 2\|I_Nf_1\|_{H^1}\|e^{ay}Ip_yf_2\|_{L^2})^2
\end{align*}
as claimed.
\end{proof}

We next recall the definition of the space $X^{s,1/2,1}$. We define the sets $A_{j}$ and $B_{k}$ by
\begin{align*}
A_{j} &:= \{ (\tau,\xi) \in \R^{2} \ \vert \  2^{j} \leq \langle \xi \rangle \leq 2^{j+1} \}, \qquad j \geq 0\\
B_{k} &:= \{ (\tau,\xi) \in \R^{2} \ \vert \  2^{k} \leq \langle \tau - \xi^{3} \rangle \leq 2^{k+1} \}, \qquad k \geq 0.
\end{align*}
For $s,b \in \R$, the space $X^{s,b,1}$ is defined to be the completion of the Schwartz class functions in the norm
\begin{equation*}
\| f \|_{X^{s,b,1}} := \left ( \sum_{j \geq 0} 2^{2sj} \left ( \sum_{k \geq 0} 2^{bk} \| \widetilde{f} \|_{L^{2}(A_{j} \cap B_{k})} \right )^{2} \right )^{1/2}.
\end{equation*}
In taking $b = 1/2$ we have the following embeddings:
\begin{equation*}
X^{s,1/2+} \hookrightarrow X^{s,1/2,1} \hookrightarrow C^{0}_{t} H^{s}_{x}.
\end{equation*}
We will work primarily in the spaces $X^{s,1/2,1}$ and $X^{s,-1/2,1}$, so we adopt the notation $X^{s} := X^{s,1/2,1}$ and $Y^{s} := X^{s,-1/2,1}$. 

The spaces $X^{s},Y^{s}$ were used in the case when $s = 1$ to prove local well-posedness for the perturbations $v$ and $w = e^{ay}v$ in $H^{1}(\R)$, see \cite{pigottraynorH1}. We review some of the features of these spaces that were used in the aforementioned local well-posedness arguments. Let $W_{1}(t)$ denote the standard Airy evolution,
\begin{equation*}
(W_{1}(t)f)\widehat{\ \ }(\xi) = e^{-it\xi^{3}} \widehat{f}(\xi).
\end{equation*}
Let $W_{2}(t)$ be the linear evolution defined for $t \geq 0$ by
\begin{equation*}
(W_{2}(t) f)\widehat{\ \ }(\xi) = e^{-it\xi^{3} - p_{a}(\xi) t} \widehat{f}(\xi),
\end{equation*}
where $p_{a}(\xi) = 3a\xi^{2} + a(c_{0}^{2} - a)$. We extend this to all of $t \in \R$ in defining
\begin{equation*}
(W_{2}(t) f)\widehat{\ \ }(\xi) = e^{-i\xi^{3}t - p_{a}(\xi)\vert t \vert} \widehat{f}(\xi).
\end{equation*}
While the Airy evolution $W_{1}(t)$ is the linear evolution associated with the unweighted perturbation $v$, the evolution $W_{2}(t)$ is the linear evolution associated with the weighted perturbation $w$. A key feature of the space $X^{s}$ is that it accommodates both of the semigroups $W_{1}(t)$ and $W_{2}(t)$, as illustrated in the following linear estimates which are valid for all $s \in \R$:
\begin{align}
\| \rho(t) W_{1}(t) f \|_{X^{s,1/2,1}} &\lesssim \| f \|_{H^{s}}, \label{W1homogeneouslinear},\\
\left \| \rho(t) \int_{0}^{t} W_{1}(t-s) F(s) ds \right \|_{X^{s,1/2,1}} &\lesssim \| F \|_{X^{s,-1/2,1}}, \label{W1nonhomogeneouslinear},\\
\intertext{and if $0 < a \leq \min(1,c_{0})$, then}
\| \rho(t) W_{2}(t) f \|_{X^{s,1/2,1}} &\lesssim \| f \|_{H^{s}}, \label{W2homogeneouslinear}\\
\left \| \chi_{\R_{+}}(t) \rho(t) \int_{0}^{t} W_{2}(t-s) F(s) ds \right \|_{X^{s,1/2,1}} &\lesssim \| F \|_{X^{s,-1/2,1}} \label{W2nonhomogeneouslinear}.
\end{align}
Here $\rho:\R \to \R$ is a cutoff function such that
\begin{equation}
\label{timecutoff}
\rho \in C^{\infty}_{0}(\R), \qquad \supp \rho \subset [-2,2], \qquad \rho \equiv 1 \ \text{on} \ [-1,1],
\end{equation}
and $\chi_{\R_{+}}$ is the indicator function for the set $\R_{+} := \{ t \in \R \ \vert \  t \geq 0 \}$. The estimates \eqref{W1homogeneouslinear}, \eqref{W1nonhomogeneouslinear} are proved in \cite{MR2501679} while the proofs of \eqref{W2homogeneouslinear}, \eqref{W2nonhomogeneouslinear} are given in \cite{pigottraynorH1}. Also crucial for the result proved in \cite{pigottraynorH1} was the following bilinear estimate, valid for all $s \geq 0$ (see Proposition 3 in \cite{pigottraynorH1}):
\begin{equation}
\label{bilinearestimate}
\| u v_{y} \|_{Y^{s}} \lesssim \| u \|_{X^{s}} \| v \|_{X^{s}}.
\end{equation}
In the case when $s = 1$ we have the following generalization of this result.

\begin{proposition}
\label{alphabilinear}
Let $\alpha_{1} \in (3/4,1], \alpha_{2} \in (0,1]$ and suppose that $u \in X^{\alpha_{1}}, v \in X^{\alpha_{2}}$. Then
\begin{equation}
\label{alphainequality}
\| u_{y} v \|_{Y^{1}} \lesssim \| u \|_{X^{\alpha_{1}}} \| v \|_{X^{\alpha_{2}}}.
\end{equation}
\end{proposition}

\begin{proof}
Since we work primarily in frequency space, we define $\widetilde{X}^{s,b,1}$ to be the completion of the Schwartz class functions in the norm
\begin{equation*}
\| f \|_{\widetilde{X}^{s,b,1}} := \left ( \sum_{j \geq 0} 2^{2sj} \left ( \sum_{k \geq 0} 2^{bk} \| f \|_{L^{2}(A_{j} \cap B_{k}) }\right )^{2} \right )^{1/2}.
\end{equation*}
Here $f = f(\tau, \xi)$ is a function of the frequency variables $\tau$ and $\xi$.
Adopting the notation $\widetilde{X}^{1} = \widetilde{X}^{1,1/2,1}$ and $\widetilde{Y}^{1} = \widetilde{X}^{1,-1/2,1}$, the estimate \eqref{alphainequality} reads
\begin{equation*}
\| ( \vert \xi_{1} \vert f ) \ast g \|_{\widetilde{Y}^{1}} \lesssim \| f \|_{\widetilde{X}^{\alpha_{1}}} \| g \|_{\widetilde{X}^{\alpha_{2}}}.
\end{equation*}

Following the proof of the standard bilinear estimate \eqref{bilinearestimate} we decompose $f$ and $g$ on dyadic blocks as follows: Define $f_{j_{1},k_{1}} := \chi_{A_{j_{1}}} \chi_{B_{k_{1}}} f$ and $g_{j_{2},k_{2}} := \chi_{A_{j_{2}}} \chi_{B_{k_{2}}} g$. We thus have
\begin{equation*}
f = \sum_{j_{1} \geq 0} \sum_{k_{1} \geq 0} f_{j_{1},k_{1}} \qquad \text{and} \qquad g = \sum_{j_{2} \geq 0} \sum_{k_{2} \geq 0} g_{j_{2},k_{2}}.
\end{equation*}
Our goal is to estimate
\begin{equation}
\label{bilinearLHS}
\sum_{j \geq 0} 2^{2j} \left ( \sum_{k \geq 0} \sum_{j_{1} \geq 0} \sum_{k_{1} \geq 0} \sum_{j_{2} \geq 0} \sum_{k_{2} \geq 0} 2^{-k/2} 2^{j_{1}} \| f_{j_{1},k_{1}} \ast g_{j_{2},k_{2}} \|_{L^{2}(A_{j} \cap B_{k})} \right )^{2}.
\end{equation}
Indeed, we wish to establish an estimate of the form
\begin{equation*}
\text{\eqref{bilinearLHS}} \lesssim \| f \|_{\widetilde{X}^{\alpha_{1}}}^{2} \| g \|_{\widetilde{X}^{\alpha_{2}}}^{2}.
\end{equation*}
To simplify the exposition we adopt the following notation:
\begin{align*}
F_{j_{1},k_{1}} &:= 2^{\alpha_{1}j_{1}} 2^{k_{1}/2} \| f_{j_{1},k_{1}} \|_{L^{2}}, \qquad \text{and}\\
G_{j_{2},k_{2}} &:= 2^{\alpha_{2} j_{2}} 2^{k_{2}/2} \| g_{j_{2},k_{2}} \|_{L^{2}}.
\end{align*}
The proof is divided into the following cases:
\begin{enumerate}
\item At least two of $j, j_{1}, j_{2}$ are less than 20.
\item $j_{1}, j_{2} \geq 20$ and $j < j_{1} -10$.
\item $j, j_{1} \geq 20, \vert j - j_{1} \vert \leq 10$.
\end{enumerate}

\smallskip

\noindent{\textbf{Case (1).}} Here we may assume that $j, j_{1}, j_{2} \leq 30$. Applying Young's inequality followed by H\"{o}lder's inequality yields
\begin{equation*}
\| f_{j_{1},k_{1}} \ast g_{j_{2},k_{2}} \|_{L^{2}} \lesssim 2^{j_{2}/2} 2^{15k_{1}/32} 2^{15k_{2}/32} \| f_{j_{1},k_{1}} \|_{L^{2}} \| g_{j_{2},k_{2}} \|_{L^{2}}.
\end{equation*}
After summing in $k$ and summing over $j$ (a finite sum), we find that
\begin{equation*}
\text{\eqref{bilinearLHS}} \lesssim \left ( \sum_{j_{1}=0}^{30} \sum_{k_{1} \geq 0} 2^{j_{1}} 2^{15k_{1}/32} \| f_{j_{1},k_{1}} \|_{L^{2}} \right )^{2} \left ( \sum_{j_{2} = 0}^{30} \sum_{k_{2} \geq 0} 2^{j_{2}/2} 2^{15k_{2}/32} \| g_{j_{2},k_{2}} \|_{L^{2}} \right )^{2}.
\end{equation*}
Note that the sum in $j_{2}$ is finite, so
\begin{align*}
& \sum_{j_{2} = 0}^{30} \sum_{k_{2} \geq 0} 2^{j_{2}/2} 2^{15k_{2}/32} \| g_{j_{2},k_{2}} \|_{L^{2}}\\
\leq & \left ( \sum_{j_{2} = 0}^{30} 2^{(1-2\alpha_{2})j_{2}} \right )^{1/2} \left ( \sum_{j_{2} = 0}^{30}  2^{2 \alpha_{2} j_{2}}\left ( \sum_{k_{2} \geq 0}  G_{j_{2},k_{2}} \right )^{2} \right )^{1/2}\\
\lesssim & \| g \|_{\widetilde{X}^{\alpha_{2}}}.
\end{align*}
A similar argument shows that
\begin{equation*}
\sum_{j_{1} = 0}^{30} \sum_{k_{1} \geq 0} 2^{j_{1}} 2^{15k_{1}/32} \| f_{j_{1},k_{1}} \|_{L^{2}} \lesssim \| f \|_{\widetilde{X}^{\alpha_{1}}},
\end{equation*}
which completes the argument.

\smallskip

\noindent{\textbf{Case (2).}} We may assume that $\vert j_{1} - j_{2} \vert \leq 1$, since otherwise $f_{j_{1}} \ast g_{j_{2}} = 0$ on $A_{j}$. For $(\tau_{1},\xi_{1}) \in A_{j_{1}} \cap B_{k_{1}}$ and $(\tau_{2},\xi_{2}) \in A_{j_{2}} \cap B_{k_{2}}$ we have
\begin{equation}
\label{cubicrelation}
(\tau_{1} + \tau_{2}) - (\xi_{1} + \xi_{2})^{3} - (\tau_{1} - \xi_{1}^{3}) - (\tau_{2} - \xi_{2}^{3}) = -3 \xi \xi_{1} \xi_{2}.
\end{equation}
It follows that $f_{j_{1},k_{1}} \ast g_{j_{2},k_{2}} = 0$ on $A_{j} \cap B_{k}$ unless $2^{k_{max}} \gtrsim 2^{j} 2^{j_{1}} 2^{j_{2}} \sim 2^{j + 2j_{1}}$ where $k_{max} = \max \{ k,k_{1},k_{2} \}$.

Suppose that $k = k_{max}$.
In order for $f_{j_{1},k_{1}} \ast g_{j_{2},k_{2}}$ to have low frequency support we require that whenever $(\tau_{1},\xi_{1}) \in \supp f_{j_{1},k_{1}}, (\tau_{2},\xi_{2}) \in \supp g_{j_{2},k_{2}}$, $\xi_{1}$ and $\xi_{2}$ must have opposite signs. It follows that $\supp f_{j_{1}}$ and $\supp g_{j_{2}}$ are separated by $K \sim 2^{j_{1}}$. In light of Lemma 3.3 in \cite{pigottraynorH1}, we thus have
\begin{equation*}
\| f_{j_{1},k_{1}} \ast g_{j_{2},k_{2}} \|_{L^{2}(A_{j} \cap B_{k})} \lesssim 2^{-j/2} 2^{-j_{1}/2} 2^{-\alpha_{1}j_{1}} 2^{-\alpha_{2}j_{2}} F_{j_{1},k_{1}} G_{j_{2},k_{2}}.
\end{equation*}
Therefore, using $2^{-k/2} \lesssim 2^{-j/2 - j_{1}}$, we have
\begin{align*}
\text{\eqref{bilinearLHS}} &\lesssim \sum_{j \geq 0} \left ( \sum_{j_{1} \geq j + 11} \sum_{k_{1} \geq 0} \sum_{j_{2} = j_{1} - 1}^{j_{1} + 1} \sum_{k_{2} \geq 0} 2^{-j_{1}/2} 2^{-\alpha_{1} j_{1}} 2^{-\alpha_{2}j_{2}} F_{j_{1},k_{1}} G_{j_{2},k_{2}} \right )^{2}\\
&\lesssim \sum_{j \geq 0} 2^{-j/2} \left ( \sum_{j_{1} \geq 0} \sum_{k_{1} \geq 0} \sum_{j_{2} \geq 0} \sum_{k_{2} \geq 0} 2^{-j_{1}/8 - \alpha_{1}j_{1}} 2^{-j_{2}/8 - \alpha_{2}j_{2}} F_{j_{1},k_{1}} G_{j_{2},k_{2}} \right )^{2}\\
&\lesssim \| f \|_{\widetilde{X}^{\alpha_{1}}}^{2} \| g \|_{\widetilde{X}^{\alpha_{2}}}^{2}.
\end{align*}

Next we suppose that $k_{1} = k_{max}$. In this case we require $2^{k_{1}} \gtrsim 2^{j + 2j_{1}}$. We apply Lemma 3.4 from \cite{pigottraynorH1} with $K \sim 2^{j_{1}}$ to see that
\begin{equation*}
\| f_{j_{1},k_{1}} \ast g_{j_{2},k_{2}} \|_{L^{2}(A_{j} \cap B_{k})} \lesssim 2^{k/2} 2^{-j_{1}} 2^{-k_{1}/2} 2^{-\alpha_{1} j_{1}} 2^{-\alpha_{2} j_{2}} F_{j_{1}, k_{1}} G_{j_{2}, k_{2}}.
\end{equation*}
Observe that
\begin{equation*}
2^{-k_{1}/2} \lesssim 2^{-k/16} 2^{-7k_{1}/16} \lesssim 2^{-k/16} 2^{-7j/16} 2^{-7j_{1}/8}.
\end{equation*}
It follows that
\begin{align*}
\text{\eqref{bilinearLHS}}
%&\lesssim \sum_{j \geq 0} \left ( \sum_{k \geq 0} \sum_{j_{1} \geq j + 11} \sum_{k_{1} \geq 0} \sum_{j_{2} = j_{1} - 1}^{j_{1} + 1} \sum_{k_{2} \geq 0} 2^{j} 2^{-k/16} 2^{-7j/16} 2^{-7j_{1}/8} 2^{-\alpha_{1} j_{1}} 2^{-\alpha_{2} j_{2}} F_{j_{1},k_{1},\alpha_{1}} G_{j_{2},k_{2},\alpha_{2}} \right )^{2}\\
&\lesssim \sum_{j \geq 0} 2^{-j/16} \left ( \sum_{\substack{j_{1} \geq 0\\ k_{1} \geq 0}}  \sum_{\substack{j_{2} \geq 0\\ k_{2} \geq 0}}  2^{-j_{1}/8} 2^{-j_{2}/8} 2^{-\alpha_{1} j_{1}} 2^{-\alpha_{2} j_{2}} F_{j_{1},k_{1}} G_{j_{2},k_{2}} \right )^{2}\\
&\lesssim \| f \|_{\widetilde{X}^{\alpha_{1}}}^{2} \| g \|_{\widetilde{X}^{\alpha_{2}}}^{2}.
\end{align*}

Finally we consider the case when $k_{2} = k_{max}$. Since the expression to be estimated is symmetric in $(j_{1},k_{1})$ and $(j_{2},k_{2})$, we can argue as in the case where $k_{1} = k_{max}$ to obtain the desired estimate.

\smallskip

\noindent{\textbf{Case (3).}} In this case we may assume that $j_{2} \leq j + 11$. In light of \eqref{cubicrelation} we require $2^{k_{max}} \gtrsim 2^{2j + j_{2}}$. We begin by assuming that $k = k_{max}$. Lemma 3.3 from \cite{pigottraynorH1} gives
\begin{equation*}
\| f_{j_{1},k_{1}} \ast g_{j_{2},k_{2}} \|_{L^{2}(A_{j} \cap B_{k})} \lesssim 2^{-j/4} 2^{-\alpha_{1} j_{1}} 2^{-\alpha_{2} j_{2}} F_{j_{1},k_{1}} G_{j_{2},k_{2}}.
\end{equation*}
Therefore, since $2^{-k/2} \lesssim 2^{-j - j_{2}/2}$, we find
\begin{align*}
\text{\eqref{bilinearLHS}} 
%&\lesssim \sum_{j \geq 0} \left ( \sum_{k \geq 2j + j_{2}} \sum_{j_{1} = j-10}^{j+10} \sum_{k_{1} = 0}^{k} \sum_{j_{2} = 0}^{j + 11} \sum_{k_{2}=0}^{k} 2^{j} 2^{j_{1}} 2^{-k/2} 2^{-j/4} 2^{-\alpha_{1} j_{1}} 2^{-\alpha_{2} j_{2}} F_{j_{1},k_{1}, \alpha_{1}} G_{j_{2},k_{2},\alpha_{2}} \right )^{2}\\
&\lesssim \sum_{j \geq 0} 2^{-j \epsilon} \left ( \sum_{j_{1} = j - 10}^{j + 10} \sum_{k_{1} \geq 0} \sum_{j_{2} = 0}^{j+11} \sum_{k_{2} \geq 0} 2^{j_{1}(\frac{3}{4} - \alpha_{1} + )} 2^{(-\alpha_{2} -1/2)j_{2}} F_{j_{1},k_{1}} G_{j_{2},k_{2}} \right )^{2}\\
&\lesssim \| f \|_{\widetilde{X}^{\alpha_{1}}}^{2} \| g \|_{\widetilde{X}^{\alpha_{2}}}^{2},
\end{align*}
provided $\alpha_{1} > 3/4$ and $\epsilon>0$ is chosen appropriately small.

Suppose that $k_{1} = k_{max}$, meaning that $2^{k_{1}} \gtrsim 2^{2j + j_{2}}$. We apply Lemma 3.4 from \cite{pigottraynorH1} to estimate
\begin{equation*}
\| f_{j_{1},k_{1}} \ast g_{j_{2},k_{2}} \|_{L^{2}(A_{j} \cap B_{k})} \lesssim 2^{k/4} 2^{-j_{1}/4} 2^{-\alpha_{1} j_{1}} 2^{-\alpha_{2} j_{2}} 2^{-k_{1}/2} F_{j_{1},k_{1}} G_{j_{2},k_{2}}.
\end{equation*}
After using $2^{-k_{1}/2} \lesssim 2^{-j-j_{2}/2}$
\begin{align*}
\text{\eqref{bilinearLHS}} 
%&\lesssim \sum_{j \geq 0} \left ( \sum_{k = 0}^{k_{1}} \sum_{j_{1} = j-10}^{j+10} \sum_{k_{1} \geq 2j + j_{2}} \sum_{j_{2} = 0}^{j + 11} \sum_{k_{2} = 0}^{k_{1}} 2^{j} 2^{j_{1} (3/4 - \alpha_{1})}  2^{-k/4} 2^{-\alpha_{2} j_{2}} 2^{-k_{1}/2} F_{j_{1},k_{1}} G_{j_{2},k_{2}} \right )^{2}\\
&\lesssim \sum_{j \geq 0} 2^{-j\epsilon} \left ( \sum_{j_{1} \geq 0} \sum_{k_{1} \geq 0} \sum_{j_{2} \geq 0} \sum_{k_{2} \geq 0} 2^{j_{1}(-\alpha_{1} + \epsilon + \frac{3}{4})} 2^{-j_{2}/2} 2^{-\alpha_{2} j_{2}} F_{j_{1},k_{1}} G_{j_{2},k_{2}} \right )^{2}\\
&\lesssim \| f \|_{\widetilde{X}^{\alpha_{1}}}^{2} \| g \|_{\widetilde{X}^{\alpha_{2}}}^{2},
\end{align*}
again provided $\alpha_{1} > 3/4$ and $\epsilon > 0$ is chosen to be sufficiently small.

Finally we consider the case for which $k_{2} = k_{max}$, so that $2^{k_{2}} \gtrsim 2^{2j + j_{2}}$. We divide our analysis into the following two subcases:
\begin{enumerate}
\item[(i)] $\vert j_{2} - j \vert \leq 5$.
\item[(ii)] $\vert j_{2} - j \vert > 5$.
\end{enumerate}
In case (i) we use Lemma 3.4 from \cite{pigottraynorH1} to estimate
\begin{equation*}
\| f_{j_{1},k_{1}} \ast g_{j_{2},k_{2}} \|_{L^{2}(A_{j} \cap B_{k})} \lesssim 2^{k/4} 2^{-j_{2}/4} 2^{-k_{2}/2} 2^{-\alpha_{1} j_{1}} 2^{-\alpha_{2} j_{2}} F_{j_{1},k_{1}} G_{j_{2},k_{2}}.
\end{equation*}
We thus obtain
\begin{align*}
\text{\eqref{bilinearLHS}}
%&\lesssim \sum_{j \geq 0} \left ( \sum_{k = 0}^{k_{2}} \sum_{j_{1} = j-10}^{j+10} \sum_{k_{1}=0}^{k_{2}} \sum_{\substack{j \geq 0\\ \vert j - j_{2} \vert \leq 5}} \sum_{k_{2} \geq 2j + j_{2}} 2^{j} 2^{j_{1}} 2^{-k/2} 2^{k/4} 2^{-j_{2}/4} 2^{-k_{2}/2} 2^{-\alpha_{1} j_{1}} 2^{-\alpha_{2} j_{2}} F_{j_{1},k_{1},\alpha_{1}} G_{j_{2},k_{2},\alpha_{2}} \right )^{2}\\
&\lesssim \sum_{j \geq 0} \left ( \sum_{j_{1} = j-10}^{j+10} \sum_{k_{1} = 0}^{k_{2}} \sum_{\substack{j_{2} \geq 0\\ \vert j - j_{2} \vert \leq 5}} \sum_{k_{2} \geq 0} 2^{j_{1}} 2^{-3j_{2}/4} 2^{-\alpha_{1} j_{1}} 2^{-\alpha_{2} j_{2}} F_{j_{1},k_{1}} G_{j_{2},k_{2}} \right )^{2}\\
&\lesssim \sum_{j \geq 0} 2^{-j/2} \left ( \sum_{j_{1} \geq 0} \sum_{k_{1} \geq 0} \sum_{j_{2} \geq 0} \sum_{k_{2} \geq 0} 2^{j_{1}(-1/4 - \alpha_{1})} 2^{j_{2}(-1/4 - \alpha_{2})} F_{j_{1},k_{1}} G_{j_{2},k_{2}} \right )^{2}\\
&\lesssim \| f \|_{\widetilde{X}^{\alpha_{1}}}^{2} \| g \|_{\widetilde{X}^{\alpha_{2}}}^{2}.
\end{align*}
In case (ii) we again use Lemma 3.4 with $K \sim 2^{j}$ to estimate
\begin{equation*}
\| f_{j_{1},k_{1}} \ast g_{j_{2},k_{2}} \|_{L^{2}(A_{j} \cap B_{k})} \lesssim 2^{k/2} 2^{-j/2} 2^{-j_{2}/2} 2^{-k_{2}/2} 2^{-\alpha_{1}j_{1}} 2^{-\alpha_{2} j_{2}} F_{j_{1},k_{1}} G_{j_{2},k_{2}}.
\end{equation*}
Next we estimate
\begin{equation*}
2^{-k_{2}/2} \lesssim 2^{-k/16} 2^{-7k_{2}/16} \lesssim 2^{-k/16} 2^{-7j/8} 2^{-7j_{2}/16}.
\end{equation*}
We thus find that
\begin{align*}
\text{\eqref{bilinearLHS}} &\lesssim \sum_{j \geq 0} \left ( \sum_{j_{1} = j-10}^{j+10} \sum_{k_{1} \geq 0} \sum_{j_{2} = 0}^{j-5} \sum_{k_{2} \geq 0} 2^{j_{1}} 2^{-3j/8} 2^{-15j_{2}/16} 2^{-\alpha_{1} j_{1}} 2^{-\alpha_{2} j_{2}} F_{j_{1},k_{1}} G_{j_{2},k_{2}} \right )^{2}\\
&\lesssim \sum_{j \geq 0} 2^{-j/8} \left ( \sum_{j_{1} \geq 0} \sum_{k_{1} \geq 0} \sum_{j_{2} \geq 0} \sum_{k_{2} \geq 0} 2^{j_{1}(-\alpha_{1} + 9/16)} 2^{j_{2}(-\alpha_{2} -15/16)}  F_{j_{1},k_{1}} G_{j_{2},k_{2}} \right )^{2}\\
&\lesssim \| f \|_{\widetilde{X}^{\alpha_{1}}}^{2} \| g \|_{\widetilde{X}^{\alpha_{2}}}^{2},
\end{align*}
since $\alpha_{1} > 3/4$.
\end{proof}

In the proof of the modified local well-posedness result we will require the following estimate.

\begin{proposition}
\label{Imethoddecay}
Let $s > 7/8$. Suppose that $u$ and $v$ are spacetime functions such that $u,v \in X^{s}$ and $e^{ay} Iu, e^{ay} Iv \in X^{1}$. Then
\begin{equation}
\label{decayestimate}
\begin{aligned}
&\left \| e^{ay} \partial_{y} \Big ( I(uv) - Iu Iv \Big ) \right \|_{Y^{1}}\\ 
\lesssim &N^{\frac{3}{4} - s +} \left ( \| e^{ay} Iu \|_{X^{1}} \| Iv \|_{X^{1}} + \| Iu \|_{X^{1}} \| e^{ay} Iv \|_{X^{1}} \right ).
\end{aligned}
\end{equation}
\end{proposition}

\begin{remark}
Since $s > 7/8$ we see that \eqref{decayestimate} implies
\begin{align*}
&\left \| e^{ay} \partial_{y} \Big ( I(uv) - Iu Iv \Big ) \right \|_{Y^{1}}\\ \lesssim &N^{-1/8 +} \left ( \| e^{ay} Iu \|_{X^{1}} \| Iv \|_{X^{1}} + \| Iu \|_{X^{1}} \| e^{ay} Iv \|_{X^{1}} \right ).
\end{align*}
\end{remark}

\begin{proof}[Proof of Proposition \ref{Imethoddecay}]
For a function $u(t,x)$ of spacetime we let $u_{N_{j}}$ denote the function whose Fourier transform is given by $\widehat{u}_{N_{j}} = \eta_{A_{j}}(\xi) \widehat{u}(\xi)$, where $\eta_{A_{j}}$ is a smooth cutoff function adapted to the set $A_{j} := \{ \xi \in \R \ \vert \  \vert \xi \vert \sim N_{j} \}$ with $N_{j}$ dyadic. 

We truncate the exponential weight using a spatial cutoff function. Specifically, for $R > 1$ we let $\vartheta_{R}:\R \to \R$ by
\begin{equation*}
\vartheta_{R}(y) = \left \{ \begin{array}{l l}
1, & y < R\\
0, & y > R,
\end{array}
\right .
\end{equation*}
and define $\omega_{a,R}(y) := \vartheta_{R}(y) e^{ay}$. Observe that $\omega_{a,R} \in H^{s}(\R)$ for all $s \in \R$; in particular, it makes sense to speak of the Fourier transform of $\omega_{a,R}$. Furthermore, we have the following approximation result.

\begin{lemma}
\label{omegaaRH1}
If $f \in H^{1}_{a}(\R)$, then
\begin{equation*}
\lim_{R \to \infty} \| \omega_{a,R} f \|_{H^{1}} = \| e^{ay} f \|_{H^{1}}.
\end{equation*}
\end{lemma}

\begin{proof}
Arguing as in the proof of Lemma \ref{expIproductrule}, we find that
\begin{equation}
\label{omegaaRconvergenceL2}
\lim_{R \to \infty} \| \omega_{a,R} f \|_{L^{2}} = \| e^{ay} f \|_{L^{2}}.
\end{equation}
Observe that $\| e^{ay} f \|_{H^{1}}^{2} = \| e^{ay} f \|_{L^{2}}^{2} + \| e^{ay} (af + f_{y}) \|_{L^{2}}^{2}$. One also checks that
\begin{equation*}
\| \omega_{a,R} f \|_{H^{1}}^{2} = \| \omega_{a,R} f \|_{L^{2}}^{2} + \| \omega_{a,R} (af + f_{y}) \|_{L^{2}}^{2}.
\end{equation*}
In light of this calculation and \eqref{omegaaRconvergenceL2}, we obtain the conclusion of the lemma.
\end{proof}

To prove \eqref{decayestimate} it suffices to show that
\begin{equation}
\label{decayestimatesufficient}
\begin{aligned}
&\| \widehat{g}_{N_{1}} \vert \xi_{2} + \xi_{3} \vert \big ( m(\xi_{2} + \xi_{3}) - m(\xi_{2}) m(\xi_{3}) \big ) \widehat{u}_{N_{2}} \widehat{v}_{N_{3}} \|_{\widetilde{Y}^{1}} \\
\lesssim & N^{\frac{3}{4} - s +} \Big ( N_{12}^{0-} N_{3}^{0-} \| g_{N_{1}} Iu_{N_{2}} \|_{X^{1}} \| Iv_{N_{3}} \|_{X^{1}}\\
& \qquad \qquad + N_{2}^{0-} N_{13}^{0-} \| Iu_{N_{2}} \|_{X^{1}} \| g_{N_{1}} Iv_{N_{3}} \|_{X^{1}} \Big ),
\end{aligned}
\end{equation}
where $g := \omega_{a,R}$.
Note that by symmetry we may assume that $N_{2} \geq N_{3}$. We adopt the notation $N_{12}$ for $\vert \xi_{1} + \xi_{2} \vert \sim N_{12}$ when $\vert \xi_{1} \vert \sim N_{1}$ and $\vert \xi_{2} \vert \sim N_{2}$. We adopt similar definitions for $N_{13}$ and $N_{23}$.

\smallskip

\noindent{\textbf{Case (1).}} $N_{2} \ll N$. In this case we see that $m(\xi_{2} + \xi_{3}) - m(\xi_{2}) m(\xi_{3}) = 0$, so the expression to be estimated vanishes.

\smallskip

\noindent{\textbf{Case (2).}} $N_{2} \gtrsim N \gg N_{3}$. We use the mean value theorem to see that
\begin{equation*}
\left \vert m(\xi_{2} + \xi_{3}) - m(\xi_{2}) m(\xi_{3}) \right \vert \lesssim \frac{N_{3}}{N_{2}} m(N_{2}) m(N_{3}).
\end{equation*}
It follows that
\begin{align*}
&\left \| g_{N_{1}} \vert \xi_{2} + \xi_{3} \vert \big ( m(\xi_{2} + \xi_{3}) - m(\xi_{2}) m(\xi_{3}) \big ) \right  \|_{\widetilde{Y}^{1}}\\
\lesssim & \frac{N_{3}}{N_{2}} \left \| \widehat{g}_{N_{1}} \vert \xi_{2} + \xi_{3} \vert \widehat{Iu}_{N_{2}} \widehat{Iv}_{N_{3}} \right \|_{\widetilde{Y}^{1}}\\
\lesssim & \frac{N_{3}}{N_{2}} \Big ( \| g_{N_{1}} Iu_{N_{2}} \partial_{y} Iv_{N_{3}} \|_{Y^{1}} + \| g_{N_{1}} Iv_{N_{3}} \partial_{y} Iu_{N_{2}} \|_{Y^{1}} \Big )\\
\lesssim & \frac{N_{3}}{N_{2}} \Big ( \| g_{N_{1}} Iu_{N_{2}} \|_{X^{3/4+}} \| Iv_{N_{3}} \|_{X^{3/4+}} + \| g_{N_{1}} Iv_{N_{3}} \|_{X^{3/4+}} \| Iu_{N_{2}} \|_{X^{3/4+}} \Big )\\
\lesssim & \frac{N_{3}}{N_{2} \langle N_{12} \rangle^{1/4-} \langle N_{3} \rangle^{1/4-}} \| g_{N_{1}} Iu_{N_{2}} \|_{X^{1}} \| Iv_{N_{3}} \|_{X^{1}}\\
& + \frac{N_{3}}{N_{2} \langle N_{13} \rangle^{1/4-} \langle N_{2} \rangle^{1/4-}} \| g_{N_{1}} Iv_{N_{3}} \|_{X^{1}} \| Iu_{N_{2}} \|_{X^{1}}
\end{align*}
Notice that
\begin{align*}
\frac{N_{3}}{N_{2} \langle N_{12} \rangle^{1/4-} \langle N_{3} \rangle^{1/4-}} \lesssim \frac{N_{3}^{3/4+}}{N_{2} \langle N_{12} \rangle^{1/4-}} \lesssim N^{-1/4+} N_{12}^{0-} N_{3}^{0-},\\
\intertext{and}
\frac{N_{3}}{N_{2} \langle N_{13} \rangle^{1/4-} \langle N_{2} \rangle^{1/4-}} \lesssim \frac{N_{3}^{3/4+}}{N_{2} \langle N_{13} \rangle^{1/4-}} \lesssim N^{-1/4+} N_{2}^{0-} N_{13}^{0-}. 
\end{align*}

\smallskip

\noindent{\textbf{Case (3).}} $N_{2} \geq N_{3} \gtrsim N$. Here we split the expression to be estimated into two terms which are then estimated separately:
\begin{align*}
&\| \widehat{g}_{N_{1}} \vert \xi_{2} + \xi_{3} \vert ( m(\xi_{2} + \xi_{3}) - m(\xi_{2}) m(\xi_{3}) ) \widehat{u}_{N_{2}} \widehat{v}_{N_{3}} \|_{\widetilde{Y}^{1}}\\
\lesssim &\| g_{N_{1}} \vert \xi_{2} + \xi_{3} \vert m(\xi_{2} + \xi_{3}) \widehat{u}_{N_{2}} \widehat{v}_{N_{3}} \|_{\widetilde{Y}^{1}}\\
+ & \| g_{N_{1}} \vert \xi_{2} + \xi_{3} \vert \widehat{Iu}_{N_{2}} \widehat{Iv}_{N_{3}} \|_{\widetilde{Y}^{1}}\\
=: & \ \text{Term I} \ + \ \text{Term II}.
\end{align*}
We estimate Term II as in Case (2) to see that
\begin{align*}
\text{Term II} &\lesssim \frac{1}{\langle N_{12} \rangle^{1/4-} \langle N_{3} \rangle^{1/4-}} \| g_{N_{1}} Iu_{N_{2}} \|_{X^{1}} \| Iv_{N_{3}} \|_{X^{1}}\\ 
&+ \frac{1}{\langle N_{13} \rangle^{1/4-} \langle N_{2} \rangle^{1/4-}} \| g_{N_{1}} Iv_{N_{3}} \|_{X^{1}} \| Iu_{N_{2}} \|_{X^{1}},
\end{align*}
which is sufficient.
Turning to Term I, we have
\begin{align*}
\text{Term I} &\lesssim m(N_{23}) \Big ( \| g_{N_{1}} u_{N_{2}} \partial_{y} v_{N_{3}} \|_{Y^{1}} + \| g_{N_{1}} v_{N_{3}} \partial_{y} u_{N_{2}} \|_{Y^{1}} \Big )\\
&\lesssim m(N_{23}) \Big ( \| g_{N_{1}} u_{N_{2}} \|_{X^{3/4+}} \| v \|_{X^{3/4+}} + \| g_{N_{1}} v_{N_{3}} \|_{X^{3/4+}} \| u_{N_{2}} \|_{X^{3/4+}} \Big )\\
&\lesssim \frac{m(N_{23})}{\langle N_{12} \rangle^{1/4-} m(N_{2}) \langle N_{3} \rangle^{s - 3/4-}} \| g_{N_{1}} Iu_{N_{2}} \|_{X^{1}} \| v_{N_{3}} \|_{X^{s}} \\
&+ \frac{m(N_{23})}{\langle N_{13} \rangle^{1/4-} m(N_{3}) \langle N_{2} \rangle^{s - 3/4-}} \| g_{N_{1}} Iv_{N_{3}} \|_{X^{1}} \| u_{N_{2}} \|_{X^{s}}\\
& \lesssim \frac{m(N_{23})}{\langle N_{12} \rangle^{1/4-} m(N_{2}) \langle N_{3} \rangle^{s - 3/4-}} \| g_{N_{1}} Iu_{N_{2}} \|_{X^{1}} \| Iv_{N_{3}} \|_{X^{1}} \\
&+ \frac{m(N_{23})}{\langle N_{13} \rangle^{1/4-} m(N_{3}) \langle N_{2} \rangle^{s - 3/4-}} \| g_{N_{1}} Iv_{N_{3}} \|_{X^{1}} \| Iu_{N_{2}} \|_{X^{1}},\\
\end{align*}
where in the final inequality we have used that $\| f \|_{X^{s}} \lesssim \| If \|_{X^{1}}$. Observe that since $N_{2} \geq N_{3}$ and $s > 3/4$ we have
\begin{equation*}
\langle N_{2} \rangle^{s - 3/4 -} m(N_{3}) \gtrsim N_{3}^{2s - 7/4 -} N^{1-s} \geq N^{s - 3/4 -},
\end{equation*}
since $s > 7/8$. It follows that
\begin{equation}
\label{TermIestimate}
\frac{m(N_{23})}{\langle N_{13} \rangle^{1/4-} m(N_{3}) \langle N_{2} \rangle^{s - 3/4 -}} \lesssim N^{\frac{3}{4} - s +} N_{13}^{0-} N_{2}^{0-}.
\end{equation}
To estimate the other multiplier expression we first note that if $N_{23} \gtrsim N_{3}$, then $m(N_{23}) \lesssim m(N_{3})$ so that
\begin{equation*}
\frac{m(N_{23})}{\langle N_{12} \rangle^{1/4-} m(N_{2}) \langle N_{3} \rangle^{s - 3/4 -}} \lesssim \frac{1}{\langle N_{12} \rangle^{1/4-} \langle N_{3} \rangle^{s - 3/4 -}},
\end{equation*}
which is acceptable. If $N_{23} \ll N_{3}$, then we must have $N_{2} \sim N_{3}$ (with the relevant factors being supported at frequencies of opposite sign), in which case may estimate $\langle N_{3} \rangle m(N_{2}) \gtrsim N^{s - 3/4 -}$. The estimate is then completed as above in \eqref{TermIestimate}.
\end{proof}

From Proposition \ref{Imethoddecay} we have the following result.

\begin{corollary}
\label{decaycorollary}
Under the hypotheses of Proposition \ref{Imethoddecay} we have
\begin{align*}
&\left \vert \int_{t_{0}}^{t_{0} + \delta} \left \langle e^{ay} Iv, e^{ay} \partial_{y} \Big ( I(uv) - Iu Iv \Big ) \right  \rangle_{H^{1}} dt \right \vert \\
\lesssim & N^{3/4 - s +} \| e^{ay} Iv \|_{X^{1}} \left ( \| e^{ay} Iu \|_{X^{1}} \| Iv \|_{X^{1}} + \| Iu \|_{X^{1}} \| e^{ay} Iv \|_{X^{1}} \right ).
\end{align*}
\end{corollary}

\begin{proof}
We apply Cauchy-Schwartz together with the embedding $X^{1,1/2+} \hookrightarrow X^{1,1/2,1}$ to see that
\begin{align*}
&\left \vert \int_{t_{0}}^{t_{0} + \delta} \langle e^{ay} Iv, e^{ay} \partial_{y} \Big ( I(uv) - Iu Iv \Big ) \rangle_{H^{1}} dt \right \vert \\
\lesssim & \| e^{ay} Iv \|_{X^{1}} \| e^{ay} \partial_{y} ( I(uv) - Iu Iv ) \|_{Y^{1}}\\
\lesssim & N^{3/4 - s +} \| e^{ay} Iv \|_{X^{1}} \Big ( \| e^{ay} Iu \|_{X^{1}} \| Iv \|_{X^{1}} + \| Iu \|_{X^{1}} \| e^{ay} Iv \|_{X^{1}} \Big ).
\end{align*}
\end{proof}

\section{Modified Local Well-Posedness}

This section is devoted to the proof of local well-posedness for the $\widetilde{v}$-equation and the $\widetilde{w}$-equation. We make the change of variables $y \mapsto y + \gamma(t) + \int_{0}^{t} c(s) ds$ and find that the initial value problem for $\widetilde{v} = I_{N}v$ is given by
\begin{equation}
\label{vlwp}
\left \{ \begin{array}{l l}
\partial_{t} \widetilde{v} + \partial_{y}^{3} \widetilde{v} + I_{N} \partial_{y}(v^{2}) +  \partial_{y}(\psi_{c} \widetilde{v}) +  I_{N}\partial_{y}(\psi_{c} v)  + (\dot{\gamma} \partial_{y} + \dot{c} \partial_{c}) I_{N} \psi_{c} = 0,\\
\widetilde{v}(0,y) = \widetilde{v}_{0}(y).
\end{array}
\right .
\end{equation}
The equation for $\widetilde{w} = e^{ay} I_{N}v$ is given by the modulation equation
\begin{equation*}
\partial_{t} \widetilde{w} = A_{a} \widetilde{w} + Q \widetilde{\mathcal{F}},
\end{equation*}
where $A_{a} = e^{ay} \partial_{y} ( - \partial_{y}^{2} + c_{0} - 2 \psi_{c}) e^{-ay}$, $Q$ is the spectral projection, and
\begin{align*}
\widetilde{\mathcal{F}} &= (c - c_{0} + \dot{\gamma})(\partial_{y} - a) \widetilde{w} - e^{ay} I_{N} \partial_{y}(v^{2}) - e^{ay}(\dot{\gamma} \partial_{y} + \dot{c} \partial_{c}) I_{N} \psi_{c}\\
&\quad - e^{ay} \partial_{y} \left ( I_{N}(\psi_{c} v) - \psi_{c} I_{N}v \right ).
\end{align*}
Upon expanding the operator $A_{a}$, we find that the initial value problem for $\widetilde{w}$ is
\begin{equation}
\label{wlwp}
\left \{ \begin{array}{l l}
\partial_{t} \widetilde{w} + \partial_{y}^{3} \widetilde{w} - 3a \partial_{y}^{2} \widetilde{w} + (3a^{2} - c_{0}) \partial_{y} \widetilde{w} + a(c_{0} - a^{2}) \widetilde{w} & \\
\quad + 2(\partial_{y} - a) (\psi_{c} \widetilde{w}) - Q \widetilde{\mathcal{F}} = 0,\\
\widetilde{w}(0,y) = \widetilde{w}_{0}(y).
\end{array}
\right .
\end{equation}

Before we proceed with our local well-posedness argument, we define the time-localized space $X^{s}_{\delta}$ to be the space with the norm
\begin{equation*}
\| u \|_{X^{s}_{\delta}} := \inf \{ \| w \|_{X^{s}} \ \vert \ w \equiv u \ \text{on}\ [0,\delta] \}.
\end{equation*}

The main goal of this section is to prove the following modified local well-posedness result:

\begin{proposition}
\label{modifiedlwp}
Let $0 < a < \sqrt{c_{0}/3}$, $s > 7/8$, and $N >1$. There is an $r > 0$ such that the following statement holds: If $v_{0} \in H^{s}(\R)$ satisfies $\| \widetilde{v}_{0} \|_{H^{1}} < r$ and $\| \widetilde{w} \|_{H^{1}} < r$ where $\widetilde{v}_{0} = I_{N} v_{0}$ and $\widetilde{w}_{0} = e^{ay} I_{N} v$, then there is a $\delta > 0$ so that the initial value problems \eqref{vlwp} and \eqref{wlwp} admit solutions $\widetilde{v}(t,y), \widetilde{w}(t,y)$, respectively, on $[0,\delta]$. Moreover these solutions satisfy
\begin{equation*}
\| \widetilde{v} \|_{X^{1}_{\delta}} \lesssim \| \widetilde{v}_{0} \|_{H^{1}}, \qquad \text{and} \qquad \| \widetilde{w} \|_{X^{1}} \lesssim \| \widetilde{w}_{0} \|_{H^{1}}.
\end{equation*}
\end{proposition}

\begin{proof}
Let $\rho:\R \to \R$ be a smooth cutoff function, as in \eqref{timecutoff}, and let $\rho_{\delta}(\cdot) = \rho(\cdot/\delta)$.
We begin by rewriting the equation for $\widetilde{v}(t,y)$, \eqref{vlwp}, using Duhamel's formula:
\begin{align*}
\widetilde{v} &= W_{1}(t) \widetilde{v}_{0} + \int_{0}^{t} W_{1}(t-s) \Big ( I_{N} \partial_{y}(v^{2}) + 2 \partial_{y}(\psi_{c} \widetilde{v}) + \partial_{y}(I_{N}(\psi_{c} v) - \psi_{c} I_{N}v ) \Big ) \\
&\quad  + \int_{0}^{t} W_{1}(t-s) (\dot{\gamma} \partial_{y} + \dot{c} \partial_{c}) I_{N}\psi_{c} ds.
\end{align*}
We will show that the map $\Phi$ given by
\begin{align*}
\Phi \widetilde{v} &:= \rho_{\delta}(t) W_{1}(t) \widetilde{v}_{0} + \rho_{\delta}(t) \int_{0}^{t} W_{1}(t-s) \Big ( I_{N} \partial_{y}(v^{2}) + 2 \partial_{y}(\psi_{c} \widetilde{v})\Big ) ds  \\
&\quad  +  \rho_{\delta}(t) \int_{0}^{t} W_{1}(t-s) \Big ( \partial_{y}(I_{N}(\psi_{c} v) - \psi_{c} I_{N}v )+ (\dot{\gamma} \partial_{y} + \dot{c} \partial_{c}) I_{N}\psi_{c} \Big ) ds
\end{align*}
is a contraction on a small ball in $X^{1}_{\delta}$. We estimate $\Phi \widetilde{v}$ in $X^{1}_{\delta}$ using \eqref{W1homogeneouslinear} and \eqref{W1nonhomogeneouslinear}:
\begin{align*}
\| \Phi \widetilde{v} \|_{X^{1}_{\delta}} &\lesssim \| \widetilde{v}_{0} \|_{H^{1}} + \| I_{N} \partial_{y}(v^{2}) \|_{Y^{1}_{\delta}} + \| \partial_{y} (\psi_{c} \widetilde{v}) \|_{Y^{1}_{\delta}}\\
&+ \| \partial_{y} ( I_{N}(\psi_{c} v) - \psi_{c} I_{N}v) \|_{Y^{1}_{\delta}} + \| (\dot{\gamma} \partial_{y} + \dot{c} \partial_{c}) I_{N} \psi_{c} \|_{Y^{1}_{\delta}}\\
&=: \| \widetilde{v}_{0} \|_{H^{1}} + \text{Term I} + \text{Term II} + \text{Term III} + \text{Term IV} .
\end{align*}
To estimate Term I we first note that
\begin{equation*}
\| I_{1} \partial_{y}(v^{2}) \|_{Y^{1}_{\delta}} \sim \| \partial_{y}(v^{2}) \|_{Y^{s}_{\delta}} \lesssim \| v \|_{X^{s}_{\delta}}^{2} \sim \| I_{1} v \|_{X^{1}_{\delta}}^{2}.
\end{equation*}
In light of Lemma 12.1 from \cite{MR2054622} we may conclude that
\begin{equation*}
\| I_{N} \partial_{y} (v^{2}) \|_{Y^{1}_{\delta}} \lesssim \| I_{N}v \|_{X^{1}_{\delta}}^{2} = \| \widetilde{v} \|_{X^{1}}^{2}.
\end{equation*}
To estimate Term II we use the bilinear estimate \eqref{bilinearestimate} to see that
\begin{equation*}
\text{Term II} \lesssim \| \psi_{c} \|_{X^{1}_{\delta}} \| \widetilde{v} \|_{X^{1}_{\delta}}.
\end{equation*}
Recall that for $\delta, \epsilon > 0$ sufficiently small we have
\begin{equation*}
\| \psi_{c} \|_{X^{1}_{\delta}} \lesssim \delta^{\epsilon}.
\end{equation*}
Thus
\begin{equation*}
\text{Term II} \lesssim \delta^{\epsilon} \| \widetilde{v} \|_{X^{1}_{\delta}}.
\end{equation*}
Turning to Term III we argue as for Terms I and II to find that
\begin{equation*}
\text{Term III} \lesssim \| \partial_{y} I_{N}(\psi_{c} v) \|_{Y^{1}_{\delta}} + \| \partial_{y}(\psi_{c} I_{N}v) \|_{Y^{1}_{\delta}} \lesssim \delta^{\epsilon} \| \widetilde{v} \|_{X^{1}_{\delta}}.
\end{equation*}
Finally, for Term IV we recall that from the modulation equations we have 
\begin{equation*}
\| \dot{\gamma} \|_{L^{\infty}_{t}}, \| \dot{c} \|_{L^{\infty}_{t}} \lesssim \| \widetilde{w} \|_{X^{1}_{\delta}}
\end{equation*}
 so that
\begin{equation*}
\text{Term IV} \lesssim \delta^{\epsilon} \| \widetilde{w} \|_{X^{1}_{\delta}}.
\end{equation*}
Taken all together we have
\begin{equation}
\label{Phivestimate}
\| \Phi \widetilde{v} \|_{X^{1}_{\delta}} \lesssim \| \widetilde{v}_{0} \|_{H^{1}} + \| \widetilde{v} \|_{X^{1}_{\delta}}^{2} + \delta^{\epsilon} \| \widetilde{v} \|_{X^{1}_{\delta}} + \delta^{\epsilon} \| \widetilde{w} \|_{X^{1}_{\delta}}.
\end{equation}

For the $\widetilde{w}$ equation we expand the spectral projection $Qf = f - \sum_{j=1}^{2} \langle f, \eta_{j} \rangle \zeta_{j}$ and make the change of variables $y \mapsto y - ( (3a^{2} - c_{0})t + \gamma(t) - \int_{0}^{t} c(s) ds)$, so that the equation for $\widetilde{w}$ reads
\begin{align*}
& \partial_{t} \widetilde{w} + \partial_{y}^{3} \widetilde{w} - 3a \partial_{y}^{2} \widetilde{w} + a(c_{0} - a^{2} - c + c_{0}) \widetilde{w} - a \dot{\gamma} \widetilde{w} - e^{ay} I_{N} \partial_{y}(v^{2})\\
&\ \  - e^{ay} (\dot{\gamma} \partial_{y} + \dot{c} \partial_{c}) I_{N} \psi_{c} - e^{ay} \partial_{y} ( I_{N}(\psi_{c} v) - \psi_{c} I_{N} v)\\
&\ \  + \langle \widetilde{\mathcal{F}}, \eta_{1} \rangle \zeta_{1} + \langle \widetilde{\mathcal{F}}, \eta_{2} \rangle \zeta_{2} = 0.
\end{align*}
Rewriting this equation using Duhamel's formula leads us to define the following operator
\begin{align*}
&\Psi \widetilde{w} = \rho_{\delta}(t)  W_{2}(t) \widetilde{w}_{0} + \rho_{\delta}(t) \int_{0}^{t} W_{2}(t-s) \Big ( 2(\partial_{y} - a) (\rho_{\delta}^{2} \psi_{c} \widetilde{w}) + a \rho_{\delta} \dot{\gamma} \widetilde{w} \Big ) ds\\
+ &\rho_{\delta}(t) \int_{0}^{t} W_{2}(t-s) \Big ( a(c - c_{0}) \rho_{\delta} \widetilde{w} - e^{ay} I_{N} \partial_{y}(\rho_{\delta}^{2} v^{2}) \Big ) ds\\
+ &\rho_{\delta}(t) \int_{0}^{t} W_{2}(t-s) \Big (- e^{ay} (\dot{\gamma} \partial_{y} + \dot{c} \partial_{c}) \rho_{\delta} I_{N} \psi_{c} + e^{ay} \partial_{y} (I_{N}(\psi_{c} v) - \psi_{c} I_{N}v) \Big ) ds\\
 + & \rho_{\delta}(t) \int_{0}^{t} W_{2}(t-s) \Big ( \rho_{\delta} \langle \widetilde{\mathcal{F}},\eta_{1} \rangle \zeta_{1} + \rho_{\delta} \langle \widetilde{\mathcal{F}},\eta_{2} \rangle \zeta_{2} \Big ) ds,
\end{align*}
%{\color{red}{Why do we require the $\rho_{\delta}$ terms inside the integral expression - check Guo and Wang?}}
which we hope to show is a contraction on a ball in $X^{1}_{\delta}$. We estimate $\Psi \widetilde{w}$ in $X^{1}_{\delta}$ using \eqref{W2homogeneouslinear} and \eqref{W2nonhomogeneouslinear}, which yields
\begin{align*}
\| \Psi \widetilde{w} \|_{X^{1}_{\delta}} &\lesssim \| \widetilde{w}_{0} \|_{H^{1}} + \| (\partial_{y} - a) \rho_{\delta}^{2} \psi_{c} \widetilde{w} \|_{Y^{1}_{\delta}} + \| \rho_{\delta} \dot{\gamma} \widetilde{w} \|_{Y^{1}_{\delta}} + \| (c - c_{0}) \rho_{\delta} \widetilde{w} \|_{Y^{1}_{\delta}}\\
&+ \| e^{ay} I_{N} \partial_{y}(\rho_{\delta}^{2} v^{2}) \|_{Y^{1}_{\delta}} + \| e^{ay} (\dot{\gamma} \partial_{y} + \dot{c} \partial_{c}) \rho_{\delta} I_{N} \psi_{c} \|_{Y^{1}_{\delta}}\\
& + \| e^{ay} \partial_{y} (I_{N}(\psi_{c} v) - \psi_{c} I_{N}v ) \|_{Y^{1}_{\delta}} + \| \rho_{\delta} \langle \widetilde{\mathcal{F}}, \eta_{1} \rangle \zeta_{1} \|_{Y^{1}_{\delta}} + \| \rho_{\delta} \langle \widetilde{\mathcal{F}}, \eta_{2} \rangle \zeta_{2} \|_{Y^{1}_{\delta}}\\
&= \| \widetilde{w}_{0} \|_{H^{1}} + \text{Term I} + \text{Term II} + \text{Term III} + \text{Term IV}\\
&\ \ \   + \text{Term V} + \text{Term VI} + \text{Term VII} + \text{Term VIII}.
\end{align*}
To estimate Term I we use $e^{ay} \partial_{y} e^{-ay} = \partial_{y} - a$, $\widetilde{v} = e^{-ay} \widetilde{w}$, and the bilinear estimate \eqref{bilinearestimate} to see that
\begin{align*}
\text{Term I} &= \| e^{ay} \partial_{y} e^{-ay} \psi_{c} \widetilde{w} \|_{Y^{1}_{\delta}} = \| e^{ay} \partial_{y} \psi_{c} \widetilde{v} \|_{Y^{1}_{\delta}}\\
&\leq \| e^{ay} \widetilde{v} \partial_{y} \psi_{c} \|_{Y^{1}_{\delta}} + \| e^{ay} \psi_{c} \partial_{y} \widetilde{v} \|_{Y^{1}_{\delta}}\\
&\lesssim \| \widetilde{w} \|_{X^{1}_{\delta}} \| \psi_{c} \|_{X^{1}_{\delta}} + \| e^{ay} \psi_{c} \|_{X^{1}_{\delta}} \| \widetilde{v} \|_{X^{1}_{\delta}}\\
&\lesssim \delta^{\epsilon} \| \widetilde{w} \|_{X^{1}_{\delta}} + \delta^{\epsilon} \| \widetilde{v} \|_{X^{1}_{\delta}}.
\end{align*}
In estimating Term II we use that $\| \dot{\gamma} \|_{L^{\infty}_{t}} \lesssim \| \widetilde{w} \|_{X^{1}_{\delta}}$, which gives
\begin{equation*}
\text{Term II} \lesssim \| \widetilde{w} \|_{X^{1}_{\delta}}^{2}.
\end{equation*}
In order to estimate Term III we note that
\begin{equation*}
\vert c(t) - c_{0} \vert \leq \int_{0}^{t} \vert \dot{c}(s) \vert ds \lesssim \int_{0}^{t} \| \widetilde{w}(s) \|_{H^{1}_{x}} ds \lesssim \| \widetilde{w} \|_{L^{1}_{t} H^{1}_{x}}.
\end{equation*}
Since we are restricted to the interval $[0,\delta]$, H\"{o}lder's inequality gives
\begin{equation*}
\vert c(t) - c_{0} \vert \leq \delta^{1/2} \| \widetilde{w} \|_{L^{2}_{t} H^{1}_{x}} \lesssim \delta^{1/2} \| \widetilde{w} \|_{X^{1}_{\delta}}.
\end{equation*}
It follows that
\begin{equation*}
\text{Term III} \lesssim \| c - c_{0} \|_{L^{\infty}_{t}} \| \widetilde{w} \|_{X^{1}_{\delta}} \lesssim \| \widetilde{w} \|_{X^{1}_{\delta}}^{2}.
\end{equation*}
To estimate Term IV we use \eqref{decayestimate} and \eqref{bilinearestimate} to see that
\begin{align*}
\text{Term IV} &\leq \| e^{ay} \partial_{y} (I_{N} (\rho_{\delta}^{2} v^{2}) - \rho_{\delta}^{2} (I_{N} v)^{2}) \|_{Y^{1}_{\delta}} + \| e^{ay} \partial_{y} (I_{N}v)^{2} \|_{Y^{1}_{\delta}}\\
&\lesssim \| e^{ay} I_{N}v \|_{X^{1}_{\delta}} \| I_{N}v \|_{X^{1}_{\delta}}\\
&= \| \widetilde{w} \|_{X^{1}_{\delta}} \| \widetilde{v} \|_{X^{1}_{\delta}}.
\end{align*}
The estimate for Term V is similar to the one we used for the analogous term in the $\widetilde{v}$ equation (term $(IV)$), yielding
\begin{equation*}
\text{Term V} \lesssim \delta^{\epsilon} \| \widetilde{w} \|_{X^{1}_{\delta}}.
\end{equation*}
Term VI is estimated using \eqref{decayestimate}, \eqref{bilinearestimate}, and the fact that $\| I_{N} \psi_{c} - \psi_{c} \|_{X^{1}_{\delta}} \lesssim N^{-C}$ with $C$ as large as need be:
\begin{align*}
\text{Term VI} &\leq \| e^{ay} \partial_{y} ( I_{N}(\psi_{c} v) - I_{N} \psi_{c} I_{N} v) \|_{Y^{1}_{\delta}} + \| e^{ay} \partial_{y} ( \psi_{c} - I_{N} \psi_{c}) I_{N}v \|_{Y^{1}_{\delta}}\\
&\lesssim N^{-1/8+} \delta^{\epsilon} \| \widetilde{v} \|_{X^{1}_{\delta}} + N^{-1/8+} \delta^{\epsilon} \| \widetilde{w} \|_{X^{1}_{\delta}} + N^{-C} \| \widetilde{v} \|_{X^{1}_{\delta}} + N^{-C} \| \widetilde{w} \|_{X^{1}_{\delta}},
\end{align*}
 leaving us with
\begin{align*}
\text{Term VI} &\lesssim \delta^{\epsilon} \| \widetilde{v} \|_{X^{1}_{\delta}} + \delta^{\epsilon} \| \widetilde{w} \|_{X^{1}_{\delta}}.
\end{align*}
Turning to Terms VII and VIII we recall from Lemma 3.5 in \cite{pigottraynorH1} that
\begin{equation*}
\| \langle f, \eta_{j} \rangle \zeta_{j} \|_{Y^{1}_{\delta}} \lesssim \| f \|_{Y^{1}_{\delta}}, \qquad j = 1,2.
\end{equation*}
It follows that
\begin{equation*}
\text{Term VII, Term VIII} \lesssim \| \widetilde{\mathcal{F}} \|_{Y^{1}_{\delta}} \lesssim \| \widetilde{w} \|_{X^{1}_{\delta}}^{2} + \| \widetilde{v} \|_{X^{1}_{\delta}} \| \widetilde{w} \|_{X^{1}_{\delta}} + \delta^{\epsilon} \| \widetilde{w} \|_{X^{1}_{\delta}} + \delta^{\epsilon} \| \widetilde{v} \|_{X^{1}_{\delta}}.
\end{equation*}
Altogether, then, we have
\begin{align*}
\| \Psi \widetilde{w} \|_{X^{1}_{\delta}} \lesssim \| \widetilde{w}_{0} \|_{H^{1}} + \delta^{\epsilon} \| \widetilde{w} \|_{X^{1}_{\delta}} + \delta^{\epsilon} \| \widetilde{v} \|_{X^{1}_{\delta}} + \| \widetilde{w} \|_{X^{1}_{\delta}}^{2} + \| \widetilde{w} \|_{X^{1}_{\delta}} \| \widetilde{v} \|_{X^{1}_{\delta}}.
\end{align*}

Suppose that $\| \widetilde{v}_{0} \|_{H^{1}}, \| \widetilde{w}_{0} \|_{H^{1}} < r \ll 1$ and let
\begin{equation*}
\mathcal{B} = \left \{ \widetilde{v}, \widetilde{w} \in X^{1}_{\delta} \ \vert \ \| \widetilde{v} \|_{X^{1}_{\delta}} \leq 2cr, \ \| \widetilde{w} \|_{X^{1}_{\delta}} \leq 2cr \right \}.
\end{equation*}
Using the estimates that we have established, it transpires that $\Phi, \Psi:\mathcal{B} \to \mathcal{B}$ are contractions following the arguments from Proposition 4 of \cite{pigottraynorH1}. The desired result follows.
\end{proof}

\section{Iteration}
In this section, we prove the main result of the paper, namely the exponential decay of the weighted perturbation given in Theorem \ref{main}.  
We will prove the result by induction.  Define $\dot{c}_n$ and $\dot{\gamma}_n$ by \eqref{modulation}, and let the variable $y$ be defined accordingly as $y= x-\int_0^tc(s)ds-\gamma(t)$.  Let $T>0$ be given.  Let $\kappa = (\max(1-b,\frac34))^{\frac{1}{2+\frac{1-s}{s-\frac34-}}-}$.  Let $N(n) = \kappa^{(-\frac{1}{\frac34-s+}+)n}$.  Now, let $\epsilon_1$ and $c_2$  be sufficiently small so that, whenever $\|e^{ay}I_{N(n)}w(t_n)\|_{H^1}<2\epsilon_1$ and $ \|I_{N(n)}v(t_n)\|_{H^1}<c_2,$ it follows that  $v(t)$ exists on $[t_0,t_0+\delta]$, and 
\begin{equation}\label{lwpest}
\|w\|_{X^{1,b}_{[t_0,t_0+\delta]}} <C_0\epsilon_1 \qquad \mbox{ and } \qquad \|v\|_{X^{1,b}_{[t_0,t_0+\delta]}} < C_0c_2,
\end{equation}
where $C_0$ is the implicit constant in the conclusion of Proposition \ref{modifiedlwp}.  Additionally, assume that $c_2 < \frac{b}{10}$.  Let $n_0=\frac{T}{\delta}$.  Finally, choose $\epsilon_2$ sufficiently small that $Cr^{\frac{n_0}{2}} \epsilon_2 < c_2$, with $r$ to be expressed later.  

We must recall the known control on $v$.  In \cite{MR3082567} it is proven that, with $H(f) = \int |\p_xf|^2-\frac23f^3$, 
\begin{align*}
\|\tilde{v}_n(n)\|_{H^1}^2 & \sim H(\psi+\tilde{v}_n(n))-\left ( \frac{\|\psi + \tilde{v}_n(n)\|_{L^2}}{\|\psi\|_{L^2}}\right )^\frac{10}{3}H(\psi)\\
& = H(\psi +\tilde{v}_n(n))-H(\psi) + (1-\left ( \frac{\|\psi + \tilde{v}_n(n)\|_{L^2}}{\|\psi\|_{L^2}}\right )^\frac{10}{3})H(\psi).
\end{align*}
Then, since $H(\psi)$ is constant and $(1-\left ( \frac{\|\psi + \tilde{v}_n(n)\|_{L^2}}{\|\psi\|_{L^2}}\right )^\frac{10}{3})$ is very small (${\mathcal{O}}(N^{-100})$, e.g.), it suffices to increment $H(\psi+\tilde{v}_n(n))$.  It is then found in \cite{MR3082567}, as in \cite{MR1995945}, that $H(\psi+\tilde{v}_n(n+1))-H(\psi+\tilde{v}_n(n)) \sim N(n)^{-1+}\|\tilde{v}_n(n)\|_{H^1}^2.$  Therefore, when we increment $\tilde{v}_n$, we obtain that 
\begin{align*}
&\|\tilde{v}_{n+1}(n+1)\|_{H^1}^2 -\|\tilde{v}_n(n)\|_{H^1}^2 \\
= &\|\tilde{v}_{n+1}(n+1)\|_{H^1}^2 -\|\tilde{v}_n(n+1)\|_{H^1}^2+\|\tilde{v}_{n}(n+1)\|_{H^1}^2 -\|\tilde{v}_n(n)\|_{H^1}^2 \\
 \lesssim &\left ( \frac{N(n+1)}{N(n)}\right)^{1-s}-1)\|\tilde{v}_n(n+1)\|_{H^1}^2 + 
\|\tilde{v}_{n}(n+1)\|_{H^1}^2 -\|\tilde{v}_n(n)\|_{H^1}^2 \\
 \lesssim &\left ( \frac{N(n+1)}{N(n)}\right)^{1-s}-1)(\|\tilde{v}_n(n+1)\|_{H^1}^2 -\|\tilde{v}_n(n)\|_{H^1}^2) + \|\tilde{v}_{n}(n+1)\|_{H^1}^2 \\
&\qquad -\|\tilde{v}_n(n)\|_{H^1}^2 +  \left ( \frac{N(n+1)}{N(n)}\right)^{1-s}-1)\|\tilde{v}_n(n)\|_{H^1}^2\\
 =  &\left ( \frac{N(n+1)}{N(n)}\right)^{1-s}(\|\tilde{v}_n(n+1)\|_{H^1}^2-\|\tilde{v}_n(n)\|_{H^1}^2) \\
&\qquad + \left ( \frac{N(n+1)}{N(n)}\right)^{1-s}-1)\|\tilde{v}_n(n)\|_{H^1}^2\\
 \leq &\left ( \frac{N(n+1)}{N(n)}\right)^{1-s}(N(n)^{-1+}\|\tilde{v}_n(n)\|_{H^1}^2 +  \left ( \frac{N(n+1)}{N(n)}\right)^{1-s}-1)\|\tilde{v}_n(n)\|_{H^1}^2\\
 = &(\kappa^{(-\frac{1-s}{\alpha + 1-s}+\eta_1)}(N(n)^{-1+}+1)-1)\|\tilde{v}_n(n)\|_{H^1}.
\end{align*}
Therefore, for $n$ large, $$\|\tilde{v}_{n+1}(n+1) \lesssim \kappa^{\frac{1-s}{\frac34-s+}+}(N(n)^{-1+}+1)\|\tilde{v}_n(n)\|_{H^1} \leq r\|\tilde{v}_n(n)\|_{H^1}^2,$$ where $r=1.01\kappa^{\frac{1-s}{\frac34-s+}+}$ is slightly larger than $1$.  Hence it follows that 
\begin{equation}
\|\tilde{v}_n(n)\|_{H^1}^2 \leq Cr^n \epsilon_2^2.
\label{tildevgrowth}
\end{equation}
Hence it follows that $\|\tilde{v}_n(t)\|_{H^1}<c_2$ on $J_n$ for $0 \leq n \leq n_0$.  

With all these preliminaries complete, we can state the induction lemma:

\begin{lemma} Define $\tilde{w}_n(t,y)=e^{ay}I_{N(n)}v(t,y)$ and $\tilde{v}_n(t,y)=I_{N(n)}v(t,y)$ on the time interval $J_n := [t_n, t_{n+1})$,where $t_n=n\delta$.  Suppose $\|\tilde{w}(0)\|_{H^1}<\epsilon_1$, $\|\tilde{v}(0)\|_{H^1}<\epsilon_2$, and $|c(0)-c_0| < \epsilon_1$.  Then, for all $n \in \N$, the following hold:
\begin{enumerate}
\item Define $c(t)$ inductively starting at $c(0)$ by $c(t)=c(t_n)+\int_{t_n}^t\dot{c}_n(t)dt$ for $t \in [t_n,t_{n+1})$, and similarly for $\gamma(t)$.  Then $\dot{c}_n$ and $\dot{\gamma}_n$ are continuous on $J_n$ for all $n$, and $c, \gamma$ are continuous functions of $t$. 
\item $|\dot{c}_n(t_n)|<C\epsilon_1\kappa^n$,
\item $|\dot{\gamma}_n(t_n)| <C \epsilon_1\kappa^n$,
\item $|c(t_n)-c_0| < C\frac{1-\kappa^n}{1-\kappa} \epsilon_1$, and
\item $\|\tilde{w}_n(t_n)\|_{H^1} < \epsilon_1 \kappa^n$,
\end{enumerate}
where $C=2\max\{(2 + \|u\|_{L^\infty}+\|p_yu\|_{L^\infty})(\|\eta_1\|_{L^2}+ \|\eta_2\|_{L^2}),C_0^\frac32,1\}.$
\end{lemma}
\begin{proof}
Note that, for $n=0$, $t=0$ and $N(0)=1$, so (4)-(5) are verified by hypothesis.  Also note that the smoothness of $\dot{c}_n$ and $\dot{\gamma}_n$ on each $J_n$ is a standard application of the implicit function theorem.  Then $c$ and $\gamma$ are continuous by construction, so (1) holds for all $n$.  Finally, we need to verify (2)-(3) at $n=0$ in order to begin the induction.  
Note that $$\begin{bmatrix} \dot{\gamma} \\ \dot{c} \end{bmatrix} = \mathcal{A}\left (\begin{bmatrix} \langle {\tilde{\mathcal{G}}}, \eta_1 \rangle_{L^2}\\ \langle {\tilde{\mathcal{G}}}, \eta_2 \rangle_{L^2}\end{bmatrix}\right ),$$ where $$\mathcal{A} = \left ( \begin{bmatrix} 1 + \langle e^{ay}(\p_y\psi_c-\p_y\psi_{c_0}),\eta_1 \rangle -\langle \tilde{w}, \p_y\eta_1 \rangle & \langle e^{ay}(\p_c\psi_c-\p_c\psi_{c_0}),\eta_1 \rangle \\ \langle  e^{ay}(\p_y\psi_c-\p_y\psi_{c_0}),\eta_2 \rangle-\langle \tilde{w}, \p_y\eta_2 \rangle & 1 + \langle e^{ay}(\p_c\psi_c-\p_c\psi_{c_0}),\eta_2 \rangle \end{bmatrix} \right )^{-1}.$$
At any time when $|c-c_0|$ and $\|\tilde{w}_n\|_{H^1}$ are sufficiently small, it follows that $\|\mathcal{A}\|\leq 2$, so that
$$\left |\begin{bmatrix} \dot{\gamma} \\ \dot{c} \end{bmatrix}\right | \leq 2 \left |\begin{bmatrix} \langle {\tilde{\mathcal{G}}}, \eta_1 \rangle_{L^2}\\ \langle {\tilde{\mathcal{G}}}, \eta_2 \rangle_{L^2}\end{bmatrix}\right |\leq 2(\max_{i=1,2}\|\eta_i\|_{H^1})\|{\tilde{\mathcal{G}}}\|_{L^2}.$$  Finally, by Lemma \ref{expIproductrule} 
\begin{align*}
\|{\tilde{\mathcal{G}}}\|_{L^2} &= \|(c-c_0)(\p_y-a)\tilde{w} -e^{ay}I(v^2)_y-e^{ay}\p_y[I(uv)-uIv]\|_{L^2}\\
& \leq |c-c_0|\|\tilde{w}\|_{H^1} + \|e^{ay}I(v^2)_y\|_{L^2} + \|e^{ay}\p_y[I(uv)-uIv]\|_{L^2}\\
& \leq |c-c_0|\|\tilde{w}\|_{H^1} + \|Iv\|_{H^1}\|e^{ay}I\p_yv\|_{L^2}+ 2\|u\|_{L^\infty}\|e^{ay}I\p_yv\|_{L^2}\\
&\qquad + \|\p_yu\|_{L^\infty}\|e^{ay}p_yIv\|_{L^2}\\
& \leq (|c-c_0|+\|Iv\|_{H^1}+2\|u\|_{L^\infty}+\|\p_yu\|_{L^\infty})\|\tilde{w}\|_{H^1}\\
& \leq (2 + 2\|u\|_{L^\infty}+\|\p_yu\|_{L^\infty})\|\tilde{w}\|_{H^1}\\
\end{align*}
so long as $|c-c_0|$ and $\|Iv\|_{H^1}$ are at most unit size.   Therefore (2)-(3) are satisfied at $t=0$ because of our assumptions on the initial data, 
given our choice of $C$ above.

It remains to make the inductive step.  Assume that, at step $n$, (1)-(5) are valid.  In order to step forward in time, we must first gain some a priori control of the various functions on the interval $J_n$.  Without loss of generality, assume $\delta \leq 1$.  Select $\eta$  so that $24C\epsilon_1 < \eta^2$ and $\eta + c_2 < \frac{1}{20}$ (and assume $\epsilon_2$ is sufficiently small to allow this).  Define $L(t)= 8 C \|\tilde{w}\|_{H^1}+|\dot{c}|+|\dot{\gamma}| + |c-c_0|.$  Note that at $t=n$, $L(n)<11C\epsilon_1 < \frac{\eta}{2}$.  Hence, by continuity, there is a $\delta_0>0$ so that $L(t) < \eta$ on $[t_n,t_n+\delta_0)$.  Let $\delta_1$ be the largest such $\delta_0$ which is at most $\delta$.  We want to show that $\delta_1 =\delta$.  Suppose not; then $\delta_1 < \delta$.  Then $L(t_n+\delta_1)=\eta$ by continuity.  Define $J=[t_n,t_n+\delta_1]$.  On $J$, as above, we have that $\dot{c} + \dot{\gamma} < C\|\tilde{w}\|_{H^1}<\frac{\eta}{6}$.  Moreover, $|c-c_0(t)| \leq |c(n)-c_0| + \delta_1\sup_J|\dot{c}| \leq 2C\epsilon_1+ \frac{\eta}{4} \leq \frac{\eta}{12}+\frac{\eta}{6}=\frac{\eta}{4}.$  Finally, we must estimate $\|\tilde{w}(t_n+\delta_1)\|_{H^1}$.  

We have:
\begin{align*}
\|\tilde{w}(t_n+\delta_1)\|_{H^1}^2& = \|\tilde{w}(t_n)\|_{H^1}^2 + \int_J \frac{d}{dt} \|\tilde{w}\|_{H^1}^2dt \\
& = \|\tilde{w}(t_n)\|_{H^1}^2 + 2\int_J \langle \tilde{w}, \tilde{w_t} \rangle_{H^1} dt\\
&= \|\tilde{w}(t_n)\|_{H^1}^2 + 2 \int_J \langle \tilde{w}, A_a\tilde{w} + Q{\mathcal{F}} \rangle_{H^1}dt \\
& \leq \epsilon_1 + 2\int_J \langle \tilde{w},A_a\tilde{w}\rangle_{H^1}dt+\int_J\langle \tilde{w},Q{\mathcal{F}}\rangle_{H^1} dt\\
& \leq \epsilon_1 -\frac{2b\eta^2}{64C^2}+\int_J\langle\tilde{w},Q{\mathcal{F}}\rangle_{H^1}dt\\
& \leq \frac{\eta^2}{20}  -\frac{2b\eta^2}{64C^2}+\int_J\langle\tilde{w},Q{\mathcal{F}}\rangle_{H^1}dt
\end{align*}
by Proposition \ref{spectral}, the inductive hypothesis, the a priori control on $\tilde{w}$ on $J$, and the fact that the length of $J$ is at most $1$.  It remains to estimate 
\begin{align*}
&\int_J\langle\tilde{w},Q{\mathcal{F}}\rangle_{H^1}dt\\  
= &\int_J \langle \tilde{w}, Q((c-c_0-\dot{\gamma})(\p_y-a)\tilde{w}-e^{ay}I_{N(n)}\p_y(v^2)
+   e^{ay} (\dot{c}\p_c + \dot{\gamma}\p_y)I_{N(n)}\psi_c\\
&\qquad  -e^{ay}\p_y(I_{N(n)}(\psi_c v)-\psi_cI_{N(n)}v))\rangle_{H^1}dt\\ 
 = &\mbox{(I)+(II)+(III)+(IV)}. 
\end{align*}
For (I), note that $Q(\p_y-a)\tilde{w}=(\p_y-a)\tilde{w}$, and $\p_y$ is anti-symmetric, so (I)=$\int_J((c-c_0)-\dot{\gamma})(-a)\|w\|_{H^1}^2dt$, which is at most $\frac{2a\eta^3}{64C^2}$, which is certainly less than $\frac{\eta}{20}$.  For (II), we have
\begin{align*}
&\int_J \langle \tilde{w}, -e^{ay}I_{N(n)}\p_y(v^2) \rangle_{H^1}dt\\ 
 = &\int_J\la \tilde{w},e^{ay}\p_y[I_{N(n)}v^2-(I_{N(n)}v)^2]\rangle_{H^1}dt + \int_J \langle \tw, e^{ay}\p_y(I_{N(n)}v)^2 \rangle_{H^1} dt \\
 \leq  &\int_J\la \tilde{w},e^{-ay}\p_y[I_{N(n)}(v^2)-(I_{N(n)}v)^2]\rangle_{H^1}dt + \|\tw\|_{X^{1,\frac12,1}}\|e^{ay}\p_y(I_{N(n)}v)^2\|_{X^{1,-\frac12,1}}\\
 \leq &2N(n)^{-\frac14}\|\tw\|_{X^{1,\frac12,1}}\|\tw\|_{X^{1,\frac12,1}}\|\tv\|_{X^{1,\frac12,1}} + \|\tw\|_{X^{1,\frac12,1}}\|\tw\|_{X^{1,\frac12,1}}\|\tv\|_{X^{1,\frac12,1}}\\
 \leq &(1 + 2N(n))^{-\frac14})\|\tw\|_{X^{1,\frac12,1}}^2\|\tv\|_{X^{1,\frac12,1}}\\
 \leq &(1+2N(n))^{-\frac14}C_0^3\epsilon_1^2c_2\\
 \leq &\frac{1}{5760}\frac{C_0^3}{C^2}\eta^2\\
 \leq &\frac{\eta^2}{20},
\end{align*}
by Corollary \ref{decaycorollary}, Proposition \ref{alphabilinear}, and the local well-posedness estimate \eqref{lwpest}.  For (III), recall that $I_{N(n)}\psi_c-\psi_c = {\mathcal{O}}(N^{-C})$ for $C$ arbitrarily large.  So,
since \\$Q(e^{ay}\p_c\psi_{c_0})=Q(e^{ay}\p_y\psi_{c_0})=0$, we have
\begin{align*}
\mbox{(III)} & = \int_J \la \tw, Q[e^{ay}(\dot{c}\p_c + \dot{\gamma}\p_y)((I_N(n)-1)[\psi_c-\psi_{c_0}]+ [\psi_c-\psi_{c_0}]+\psi_{c_0}))]\rangle_{H^1}dt \\
& =  \int_J \la \tw, Q[e^{ay}(\dot{c}\p_c+\dot{\gamma}\p_y)((I_N(n)-1)[\psi_c-\psi_{c_0}]+ [\psi_c-\psi_{c_0}])]dt \\
& \leq (1+\tilde{C}N^{-\tilde{\tilde{C}}})\int_J(|\dot{c}|+|\dot{\gamma}|)|c-c_0|\|\tw\|_{H^1}dt\\
& \leq \tilde{C}\frac{\eta}{4}\frac{\eta}{3}\frac{\eta}{8C}\\
& \leq \frac{\eta^2}{20}.
\end{align*}
Finally, for (IV), we have 
\begin{align*}
\int_J \la & \tw, e^{ay} \p_y(I_{N(n)}(\psi_cv) -\psi_cI_{N(n)}v))\rangle_{H^1}dt \\ & = \int_J \la \tw, e^{ay}\p_y[I_{N(n)}(\psi_cv) -(I_{N(n)}\psi_c)(I_{N(n)}v]dt\\
&\qquad  +  \int_J \la \tw, e^{ay}\p_y[((I_{N(n)}\psi_c)-\psi_c)(I_{N(n)}v]dt\\
& \leq \|\tw\|_{{X^{1,\frac12,1}}}N^{-\frac14}(\|e^{ay}I_{N(n)}\psi_c\|_{{X^{1,\frac12,1}}}\|\tv\|_{{X^{1,\frac12,1}}} + \|I_{N(n)}\psi_c\|_{{X^{1,\frac12,1}}}\|\tw\|_{{X^{1,\frac12,1}}}) \\ \qquad & \qquad + \tilde{C}N^{-\tilde{\tilde{C}}}\|\tw\|_{{X^{1,\frac12,1}}}[\|e^{ay}\psi_c\|_{{X^{1,\frac12,1}}}\|\tv\|_{{X^{1,\frac12,1}}} + \|\psi_c\|_{{X^{1,\frac12,1}}}\|\tw\|_{{X^{1,\frac12,1}}}]\\
& \leq 4\|\tw\|_{{X^{1,\frac12,1}}}(N^{-\frac14}\|\tv\|_{{X^{1,\frac12,1}}}+ \|\tw\|_{{X^{1,\frac12,1}}}))\\
& \leq 4 C_0\epsilon_1(c_2 + \eta)\\
& \leq \frac{1}{120}\frac{C_0}{C}\eta^2\\
& \leq \frac{\eta^2}{20}
\end{align*}
Adding it all together, we get that $$\|\tw(t_n + \delta_1)\|_{H^1}^2 \leq \frac{\eta^2}{20} -\frac{2b\eta^2}{64C^2}+\frac{\eta^2}{20}+\frac{\eta^2}{20}+\frac{\eta^2}{20}+\frac{\eta^2}{20} < \frac{\eta^2}{4},$$ so, $L(t_n+\delta_1) < \frac{\eta}{4}+\frac{\eta}{4}+\frac{\eta}{2} = \eta,$and hence $\delta_1 = \delta$.  

Now we are ready to make the inductive step.  Consider (2)-(5) at time $t_{n+1}$.  As above, we have that $|\dot{c}_n(t_{n+1})| + |\dot{\gamma}_n(t_{n+1})| \leq 2C\|\tw(t_{n+1})\|_{H^1},$ so (2) and (3) are validated whenever (5) is.  Indeed, the estimates (2)-(3) hold on the entire interval $J_n$ whenever $\|w\|_{H^1}$ is similarly controlled on the interval.  Similarly, whenever $(2)$ is valid on $J_n$, we have
\begin{align*}
|c(t_{n+1})-c_0| & \leq |c(t_n)-c_0| + \int_{J_n} |\dot{c}_n(t)|dt \\
& \leq C\frac{1-\kappa^n}{1-\kappa}\epsilon_1 + C\kappa^n\epsilon_1\\
& \leq C\frac{1-\kappa^{n+1})}{1-\kappa}\epsilon_1,
\end{align*}
so (4) is also validated.  It therefore remains only to control $\|w_n(t)\|_{H^1}$ on $J_n$ and estimate $\|w_{n+1}(n+1)\|_{H^1}^2-\|w_n(n)\|_{H^1}^2$.  We must therefore do two things: Estimate $\|w_{n+1}(n+1)\|_{H^1}^2-\|w_n(n+1)\|_{H^1}^2$, and estimate $\|w_n(t)\|_{H^1}^2$ on $J_n$.   In what follows, for notational simplicity, we will estimate $\|w_n(t_{n+1})\|_{H^1}^2$, but the same estimate is valid for any $t \in J_n$.  Define $K_n(n)=\|w_n(t_n)\|_{H^1}^2$.  Then, as computed above, we have the following increment: 
\begin{align*}
&K_n(n+1)-K_n(n)\\ 
= &2\int_{J_n} \langle \tilde{w},A_a\tilde{w}\rangle_{H^1}dt+\int_{J_n}\langle \tilde{w},Q{\mathcal{F}}\rangle_{H^1} dt\\
= &2\int_{J_n} \langle \tilde{w},A_a\tilde{w}\rangle_{H^1}dt+2\int_{J_n} \langle \tilde{w}, Q((c-c_0-\dot{\gamma})(\p_y-a)\tilde{w}-e^{ay}I_{N(n)}\p_y(v^2) \\& \phantom{blah}+   e^{ay} (\dot{c}\p_c + \dot{\gamma}\p_y)I_{N(n)}\psi_c-e^{ay}\p_y(I_{N(n)}(\psi_c v)-\psi_cI_{N(n)}v))\rangle_{H^1}dt\\
= &\mbox{(0)+(I)+(II)+(III)+(IV)}
\end{align*}
We estimate these terms as above.  For (0), by Proposition \ref{spectral}, this is at most $-2b\int_{J_n}\|w\|_{H^1}^2dt$.  For (I), we get 
$$\int_{J_n}(c-c_0)-\dot{\gamma})(-a)\|w\|_{H^1}^2dt \leq 4a\eta \int_{J_n}\|w(t)\|_{H^1}^2dt.$$  For (II), we obtain, as above, $$\int_{J_n} \langle \tilde{w}_n, e^{ay}I_{N(n)}\p_yv^2\rangle_H^1 dt   \leq (1 + 2N(n))^{-\frac18})\|\tw_n\|_{X^{1,\frac12,1}}^2\|\tv_n\|_{X^{1,\frac12,1}} \leq Cc_0N(n).$$  Then, for (III), we get as above
$$(III) \leq (1+ \tilde{C}N^{-\tilde{\tilde{C}}})\int_J(|\dot{c}|+|\dot{\gamma}|)|c-c_0|\|\tw\|_{H^1}dt \leq 2\int_{J_n}\eta \|\tw_n(t)\|_{H^1}^2dt.$$

Finally, for (IV), we have, as above, with $\tau$ a small positive number, 
\begin{align*}
(IV) &\leq \|\tilde{w}\|_{X^{1,\frac12,1}}\left ( (N^{\frac34-s+} + \tau)\|\tw_n\|_{X^{1,\frac12,1}} + N^{\frac34-s+}\|v\|_{X^{1,\frac12,1}}\right )\\ 
&\leq 2\tau N(n) + N^{\frac34-s+}c_0\sqrt{N(n)}.
\end{align*}
Notice that $N(n)$ has been chosen so that $N(n)^{\frac34-s+} \ll \kappa^n \leq C\epsilon_1\kappa^n$.  Therefore, putting everything together, we have that 
\begin{align*}
K_n(n+1)-K_n(n) &\leq (-2b +4a\eta +2\eta)\int_{J_n}\|\tw_n(t)\|_{H^1}^2dt\\ 
&\qquad + (Cc_0 +2\tau) N(n)+Cc_0\epsilon_1\kappa^n\sqrt{K_n(n)}.
\end{align*}

Now, suppose that $K_n(n)\sim (\epsilon_1\kappa^n)^2.$  Then by the same argument as in \cite{pigottraynorH1}, it follows that $K_n(n+1) \leq \max{\{(1-b),\frac34\}}K_n(n)\leq \kappa^{2+\frac{1-s}{s-\frac34-}-}K_n(n)$.  Finally, it remains to compare $K_{n+1}(n+1)$ to $K_n(n+1)$.  By properties of the $I_N$ multiplier, we have that 
\begin{align*} K_{n+1}(n+1) & \leq \left (\frac{N(n+1)}{N(n)} \right )^{1-s}K_n(n+1) \\ & \leq \kappa^{\frac{1-s}{\frac34-s+}+}K_n(n+1) \\ & \leq  \kappa^{\frac{1-s}{\frac34-s+}+}\kappa^{2 + \frac{1-s}{s-\frac34-}-}K_n(n) \\ & \leq \kappa^2K_n(n).\end{align*}

On the other hand, if $K_n(n)\ll (\epsilon_1\kappa^n)^2,$ then the largest term on the right hand side is the last one, and we obtain that $K_n(n+1) \ll (\epsilon_1\kappa^{n})^2$.  Then $K_{n+1}(n+1) \ll \kappa^(\frac{1-s}{\frac34-s+}+)(\epsilon_1\kappa^n)^2,$ which can be taken to be at most $\epsilon_1^2\kappa^{2(n+1)}.$  In either case, after applying the inductive hypothesis, we obtain that $K_{n+1}(n+1) \leq (\epsilon_1 \kappa^{n+1})^2,$ so $\|\tilde{w}_{n+1}(n+1)\|_{H^1} \leq \epsilon_1\kappa^{n+1}$.   Hence the inductive step holds and the proof of the lemma is complete.  

\end{proof}

To conclude the proof of Theorem \ref{main}, let $r=\kappa^{\frac{1}{\delta}}$.  Then (2) and (3) are immediate from the lemma.  To conclude (1), note that \\ $\|e^{ay}I_1v(t)\|_{H^1} \leq \|e^{ay}I_Nv(t)\|_{H^1} = \|\tilde{w}(t)\|_{H^1}$ for any $N$, by the properties of $I_N$ and Lemma \ref{omegaaRH1}.  Hence (1) follows from the last conclusion of the inductive lemma.

% The bibliography.
%\bibliographystyle{plain} % (change according to your preference)
%% ***   Set the bibliography file.   ***
%\bibliography{bdreferences}

\end{document}